\documentclass[12pt]{amsart}
\usepackage{verbatim}
\usepackage{amscd}

\numberwithin{equation}{section}
\input epsf

\begin{document}

\renewcommand{\theequation}{\thesection.\arabic{equation}}
\setcounter{secnumdepth}{2}
\newtheorem{theorem}{Theorem}[section]
\newtheorem{definition}[theorem]{Definition}
\newtheorem{lemma}[theorem]{Lemma}
\newtheorem{corollary}[theorem]{Corollary}
\newtheorem{proposition}[theorem]{Proposition}
\numberwithin{equation}{section}
\theoremstyle{definition}
\newtheorem{example}[theorem]{Example}
\title[Toric generalized K$\ddot{a}$hler structures. II]
{Toric generalized K$\ddot{a}$hler structures. II}

\author[Yicao Wang]{Yicao Wang}
\address
{Department of Mathematics, Hohai University, Nanjing 210098, China}
\maketitle

\baselineskip= 20pt
\begin{abstract}
Anti-diagonal toric generalized K$\ddot{a}$hler structures of symplectic type on a compact toric symplectic manifold were investigated in \cite{Wang2} . In this article, we consider \emph{general} toric generalized K$\ddot{a}$hler  structures of symplectic type, without requiring them to be anti-diagonal. Such a structure is characterized by a triple $(\tau, C, F)$ where $\tau$ is a strictly convex function defined in the interior of the moment polytope $\Delta$ and $C, F$ are two constant anti-symmetric matrices. We prove that underlying each such a structure is a \emph{canonical} toric K$\ddot{a}$hler structure $I_0$ whose symplectic potential is given by this $\tau$, and when $C=0$ the generalized complex structure $\mathbb{J}_1$ other than the symplectic one arises from an $I_0$-holomorphic Poisson structure $\beta$ in a \emph{novel} way not mentioned in the literature before. Conversely, given a toric K$\ddot{a}$hler structure with symplectic potential $\tau$ and two anti-symmetric constant matrices $C, F$, the triple $(\tau, C, F)$ then determines a toric generalized K$\ddot{a}$hler structure of symplectic type canonically if $F$ satisfies additionally a certain positive-definiteness condition. In particular, if the initial toric K$\ddot{a}$hler manifold is the standard $M_\Delta$ associated to a Delzant polytope $\Delta$, the resulting generalized K$\ddot{a}$hler structure can be interpreted as obtained via generalized K$\ddot{a}$hler reduction from a generalized K$\ddot{a}$hler structure on an open subset of a complex linear space, just as in Delzant's construction $M_\Delta$ is obtained through K$\ddot{a}$hler reduction from a complex linear space.
\end{abstract}
\section{Introduction}
Generalized K$\ddot{a}$hler (GK) structures in generalized complex (GC) geometry are a generalization of K$\ddot{a}$hler structures in complex geometry. M. Gualtieri proved in \cite{Gu00} a remarkable result that such a structure is equivalent to the biHermitian structure first recognized by physicists trying to find the most general 2-dimensional $N=(2,2)$ supersymmetric $\sigma$-models \cite{Ga}.

Compared with K$\ddot{a}$hler geometry, GK geometry is still not a well-developed discipline and even constructing a nontrivial GK structure needs some effort. Perhaps it helps to study some kinds of simple examples first. In K$\ddot{a}$hler geometry, toric K$\ddot{a}$hler structures are well-understood mainly through the work of V. Guillemin \cite{Gul} and M. Abreu \cite{Ab}. A toric K$\ddot{a}$hler structure can be efficiently described by a strictly convex function $\tau$ defined in the interior $\mathring{\Delta}$ of the moment polytope $\Delta$. In the literature, this $\tau$ is referred to as the \emph{symplectic potential} of the toric K$\ddot{a}$hler structure. Often toric K$\ddot{a}$hler structures provide computable examples to shed some light on abstract ideas in K$\ddot{a}$hler geometry. The basic goal of \cite{Bou} \cite{Wang2} and this article as well is to extend the Abreu-Guillemin theory to the context of GK geometry. We hope this study would provide many interesting yet simple examples for GK geometry.

 In \cite{Bou} L. Boulanger started to study toric GK structures of symplectic type on a compact toric symplectic manifold $(M, \Omega, \mathbb{T}, \mu)$ ($\mathbb{T}$ is a torus acting on $M$ in an effective and Hamiltonian fashion, $\mu$ the moment map and the symplectic form $\Omega$ provides one of the two underlying GC structures); in particular, he identified a special class of such structures called \emph{anti-diagonal} ones and found that each such a structure can be characterized by a pair $(\tau, C)$, where $\tau$ is \emph{again} a strictly convex function on $\mathring{\Delta}$ and $C$ is an anti-symmetric constant matrix.

  Anti-diagonal toric GK structures of symplectic type were further explored in \cite{Wang2}. It was found that the above $\tau$ is always the symplectic potential of a canonically associated toric K$\ddot{a}$hler structure and $C$ provides a holomorphic Poisson structure $\beta$ such that the other GC structure besides the symplectic one is induced from this $\beta$ up to B-transform. In this article, we continue to study toric GK structures of symplectic type that are not necessarily anti-diagonal. Note that a key ingredient in the approach of \cite{Wang2} towards anti-diagonal toric GK structures of symplectic type is to realize that the $\mathbb{T}$-action is strong Hamiltonian in the sense of \cite{Wang} and thus can be generalized complexified. However, for the most general case, the torus action fails to be strong Hamiltonian and the geometry becomes much more complicated.

It turns out in this article that a general toric GK structure of symplectic type can be characterized by a triple $(\tau, C, F)$, where $\tau$ is the symplectic potential of a \emph{canonically} associated toric K$\ddot{a}$hler structure and $C, F$ are two constant anti-symmetric matrices. If $F=0$, we specialize to the anti-diagonal case, and if $C=F=0$, this is the classical toric K$\ddot{a}$hler case. The role of this new matrix $F$ needs to be clarified. Note that $\mu: \mathring{M}\rightarrow \mathring{\Delta}$ is a trivial principal $\mathbb{T}$-bundle over $\mathring{\Delta}$ where $\mathring{M}=\mu^{-1}(\mathring{\Delta})$. While for the anti-diagonal case only \emph{one} flat connection on $\mathring{M}$ is involved, in the general case \emph{three} flat connections arise naturally and they are related to each other by $F$. If we interpret $F$ as a deformation of the underlying canonically associated anti-diagonal toric GK structure of symplectic type, it can be imagined that before deformation, the three connections coincide and as the deformation starts, they become separated: one of them stays unchanged and the other two change in opposite directions.

To understand the different roles of $C$ and $F$ properly, let us resort to a simplified picture. Imagine how one defines a linear complex structure $I$ in a real vector space $V$. He can choose a basis $\{f_i\}$ of $V$ and a certain matrix $A$ claimed to be the matrix form of $I$ w.r.t. $\{f_i\}$. Now if he wants to deform $I$ to obtain new ones, then there are basically two ways to achieve this: on one side he can fix the basis and deform the matrix $A$, while on the other side, he can also fix the matrix but deform the basis. If we interpret $C$, $F$ as small deformations of the complex structure $I_0$ of the canonically associated toric K$\ddot{a}$hler structure , then $C$ corresponds to the first way and $F$ to the second. This explanation will be much clearer in the main body of this article.

The above investigation suggests the possibility of constructing toric GK structures from toric K$\ddot{a}$hler structures by inputting additionally two constant anti-symmetric matrices $C$ and $F$. In this aspect, $C$ and $F$ again behave very different. To realize this construction, there is no requirement on the magnitude of $C$ (this is the same as in the anti-diagonal case which was proved in \cite{Wang2}) and all feasible $C$'s form a real linear space, but $F$ must be chosen to satisfy a certain positive-definiteness condition so that all possible $F$'s only constitute a bounded set.

The rest of this article is organized as follows. \S~\ref{back} is a modest review of the necessary background material from GC geometry. \S~\ref{AG} is a brief account of Abreu-Guillemin theory and its generalization in \cite{Bou} \cite{Wang2}. Our study on general toric GK structures of symplectic type in this article really starts from \S~\ref{Log}. Basing on some essential remarks on a theorem in \cite{Wang2}, we identify a bunch of (almost) complex structures naturally arising on $\mathring{M}$. In particular, we prove that points in $\mathring{M}$ are all regular for $\mathbb{J}_1$ (the GC structure other than the symplectic one). \S~\ref{Com} is devoted to proving that the several (almost) complex structures can all be extended smoothly to the whole of $M$, in particular establishing the conclusion that underlying a toric GK structure of symplectic type there is a canonical toric K$\ddot{a}$hler structure $I_0$ (Thm.~\ref{M1} and Cor.~\ref{cano}). Since the most general case seems a bit complicated, in \S~\ref{symm} we consider the special case of $C=0$ and $F\neq 0$ (called symmetric for the obvious reason). This is the case essentially missed in \cite{Bou}. In this situation, an astonishing result is that the Hitchin Poisson structure underlying the GK structure is also the imaginary part of an $I_0$-holomorphic Poisson structure. This leads to a theorem stating that the GC structure $\mathbb{J}_1$ in the symmetric case actually arises in a novel way from $\mathbb{J}_\beta$ induced from an $I_0$-holomorphic Poisson structure $\beta$ (Thm.~\ref{Symm}). \S~\ref{CON} is devoted to proving that, given a toric K$\ddot{a}$hler structure and two anti-symmetric matrix $C, F$ such that $F$ satisfies a certain positive-definiteness condition, there is a canonical toric GK structure of symplectic type constructed from these data (Thm.~\ref{cons}). As a byproduct of the proof, we show in a global fashion how a general toric GK structure of symplectic type can be obtained from an anti-diagonal one. \S~\ref{Ex} contains an explicit example on the toric K$\ddot{a}$hler manifold $\mathbb{C}P^1\times \mathbb{C}P^1$ to demonstrate the construction mentioned above. For a given Delzant polytope $\Delta$, there is a canonical toric K$\ddot{a}$hler manifold $M_\Delta$ constructed by T. Delzant \cite{Del}. If one applies Thm.~\ref{cons} to this manifold, then the resulting toric GK structure can be interpreted as obtained from GK reduction of a toric GK structure on an open subset of $\mathbb{C}^d$, where $d$ is the number of faces of $\Delta$ of codimension 1, just as in the Delzant construction the toric K$\ddot{a}$hler manifold $M_\Delta$ is obtained from K$\ddot{a}$hler reduction of $\mathbb{C}^d$ by a certain Hamiltonian torus action. The last section is an appendix containing some facts on matrices which are frequently (maybe implicitly) used in the main body of this article.
\section{GK structures of symplectic type}\label{back}
In this section, we collect the most relevant material from GC geometry. Our basic references are \cite{Gu00} \cite{Gu0}.

A Courant algebroid $E$ is a real vector bundle $E$ over a smooth manifold $M$, together with an anchor map $\pi$ to $TM$, a non-degenerate inner product $(\cdot, \cdot)$ and a so-called Courant bracket $[\cdot , \cdot]_c$ on $\Gamma(E)$. These structures should satisfy some compatibility axioms we won't review here. $E$ is called exact, if the short sequence \[0\longrightarrow T^*M\stackrel{\pi^*}\longrightarrow E \stackrel{\pi}\longrightarrow TM \longrightarrow0\]
is exact. We only deal with exact Courant algebroids throughout this article. Given $E$, one can always find an isotropic right splitting $s:TM\rightarrow E$, with a curvature form $H\in \Omega_{cl}^3(M)$ defined by
\[H(X,Y,Z)=([s(X),s(Y)]_c,s(Z)),\quad X, Y, Z\in \Gamma(TM).\]
  By the bundle isomorphism $s+\pi^*:TM\oplus T^*M\rightarrow E$, the Courant algebroid structure can be transported onto $TM\oplus T^*M$. Then the inner product $(\cdot,\cdot)$ is the natural pairing, i.e.
$( X+\xi,Y+\eta)=\xi(Y)+\eta(X)$, and the Courant bracket is
\begin{equation*}[X+\xi, Y+\eta]_H=[X,Y]+\mathcal{L}_X\eta-\iota_Yd\xi+\iota_Y\iota_XH.\end{equation*}
Different splittings are related by B-tranforms: $e^B(X+\xi)=X+\xi+B(X)$, where $B$ is a 2-form on $M$.
\begin{definition}
 A GC structure on a Courant algebroid $E$ is a complex structure $\mathbb{J}$ on $E$ orthogonal w.r.t. the inner product and its $\sqrt{-1}$-eigenbundle $L\subset E_\mathbb{C}$ is involutive under the Courant bracket. We also say $\mathbb{J}$ is integrable in this case.
\end{definition}
For $H\equiv0$, ordinary complex and symplectic structures are extreme examples of GC structures. Precisely, for a complex structure $I$ and a symplectic structure $\Omega$, the corresponding GC structures are of the following form:
\[\mathbb{J}_I=\left(
                 \begin{array}{cc}
                   -I & 0 \\
                   0 & I^* \\
                 \end{array}
               \right),\quad \mathbb{J}_\Omega=\left(
                                                 \begin{array}{cc}
                                                   0 & \Omega^{-1} \\
                                                   -\Omega & 0 \\
                                                 \end{array}
                                               \right).
\]
 A nontrivial example beyond these is provided by a holomorphic Poisson structure $\beta$: Let $\beta$ be a holomorphic Poisson structure relative to a complex structure $J$ on $M$. Then
\[\mathbb{J}_\beta=\left(
      \begin{array}{cc}
        -J & -4\textup{Im}\beta \\
        0 & J^* \\
      \end{array}
    \right),
\]
is a GC structure, where $\textup{Im}\beta$ is the imaginary part of $\beta$.
\begin{definition}A generalized metric on a Courant algebroid $E$ is an orthogonal, self-adjoint operator $\mathcal{G}$ such that $( \mathcal{G}\cdot,\cdot)$ is positive-definite on $E$.
 \end{definition}
 A generalized metric induces a \emph{canonical} isotropic splitting: $E=\mathcal{G}(T^*M)\oplus T^*M$. It is called \emph{the metric splitting}. Given a generalized metric, we shall always choose its metric splitting to identify $E$ with $TM\oplus T^*M$. Then $\mathcal{G}$ is of the form $\left(\begin{array}{cc} 0 & g^{-1} \\g & 0 \\
\end{array} \right)$ where $g$ is an ordinary Riemannian metric.

  A generalized metric is an ingredient of a GK structure.
\begin{definition}
A GK structure on $E$ is a pair of commuting GC structures $(\mathbb{J}_1,\mathbb{J}_2)$ such that $\mathcal{G}=-\mathbb{J}_1 \mathbb{J}_2$ is a generalized metric.
\end{definition}
 A GK structure can be reformulated in many different ways, the basic of which is the biHermitian one: There are two complex structures $J_\pm$ on $M$ compatible with the metric $g$ induced from the generalized metric. Let $\omega_\pm=gJ_\pm$. Then \emph{in the metric splitting} the GC structures and the corresponding biHermitian data are related by the Gualtieri map:
 \[\mathbb{J}_1=\frac{1}{2}\left(
  \begin{array}{cc}
    -J_+-J_-& \omega_+^{-1}-\omega_-^{-1} \\
    -\omega_++\omega_- & J_+^*+J_-^* \\
  \end{array}
\right),\quad \mathbb{J}_2=\frac{1}{2}\left(
  \begin{array}{cc}
    -J_++J_-& \omega_+^{-1}+\omega_-^{-1} \\
    -\omega_+-\omega_- & J_+^*-J_-^* \\
  \end{array}
\right).\]
Note that $\beta_1:=-\frac{1}{2}(J_+-J_-)g^{-1}$ and $\beta_2:=-\frac{1}{2}(J_++J_-)g^{-1}$ are actually real Poisson structures associated to $\mathbb{J}_{1}$ and $\mathbb{J}_{2}$ respectively. As was noted by N. Hitchin in \cite{Hi}, there is a \emph{third} Poisson structure $\beta_3=\frac{1}{4}[J_+,J_-]g^{-1}$. $\beta_3$ is the common imaginary part of a $J_+$-holomorphic Poisson structure $\beta_+$ and a $J_-$-holomorphic Poisson structure $\beta_-$.

If $\mathbb{J}_2$ is a B-transform of a GC structure $\mathbb{J}_\Omega$ induced from a symplectic form $\Omega$, the GK manifold $(M, \mathbb{J}_1, \mathbb{J}_2)$ is said to be \emph{of symplectic type}. It is known from \cite{En} that for a given symplectic manifold $(M, \Omega)$, compatible GC structures $\mathbb{J}_1$ which, together with a B-transform of $\mathbb{J}_\Omega$, form GK structures on $M$ are in \emph{one-to-one} correspondence with \emph{tamed} integrable complex structures $J_+$ on $M$ whose \emph{symplectic adjoint} $J^{\Omega}:=-\Omega^{-1}J_+^*\Omega$ is also integrable. This fact greatly facilitates the study of such structures. Precisely, if we set
\[\frac{1}{2}\left(
  \begin{array}{cc}
    -J_++J_-& \omega_+^{-1}+\omega_-^{-1} \\
    -\omega_+-\omega_- & J_+^*-J_-^* \\
  \end{array}
\right)=\left(\begin{array}{cc}
1 & 0\\
-b & 1\\\end{array}\right)\left(\begin{array}{cc}
0 & \Omega^{-1}\\
-\Omega & 0\\\end{array}\right)\left(\begin{array}{cc}
1 & 0\\
b & 1\\\end{array}\right),\]
then the following basic identities can be easily obtained:
\begin{equation}J_-=J_+^\Omega=-\Omega^{-1} J_+^*\Omega,\quad g=-\frac{1}{2}\Omega (J_++J_-),\quad b=-\frac{1}{2}\Omega (J_+-J_-).\label{sy}\end{equation}
Recall that $J_+$ is tamed with $\Omega$ in the sense that the symmetric part of $-\Omega J_+$ is a Riemannian metric on $M$. Using the fact that
\[(J_++J_-)(J_+-J_-)=-(J_+-J_-)(J_++J_-)=-[J_+, J_-],\]
one can easily derive in this setting that $\beta_3=-1/2(J_+-J_-)\Omega^{-1}$.
\section{Local theory}
\subsection{Abreu-Guillemin theory and Boulanger's generalization}\label{AG}
Let us recall briefly the Abreu-Guillemin theory and its generalization in \cite{Bou} \cite{Wang2} first.
\begin{definition}A toric symplectic manifold $(M, \Omega, \mathbb{T}, \mu)$ of dimension $2n$ is a symplectic manifold $(M, \Omega)$ with an effective and Hamiltonian action of the $n$-dimensional torus $\mathbb{T}=\mathbb{T}^n$. Note that here $\mu$ is the moment map.
\end{definition}
Let $(M,\Omega, \mathbb{T}, \mu)$ be a \emph{compact} toric symplectic manifold and $\mathfrak{t}\cong \mathbb{R}^n$ the Lie algebra of $\mathbb{T}$. By the famous convexity theorem of Atiyah-Guillemin-Sternberg \cite{At} \cite{GS}, the image $\Delta$ of $\mu$ is a polytope in $\mathfrak{t}^*=(\mathbb{R}^n)^*$ which is the convex hull of the image of fixed points of the torus action. $\Delta$ is thus called the \emph{moment polytope}. In a famous theorem, Delzant proved that compact toric symplectic manifolds are classified by their moment polytopes $\Delta$ up to equivariant symplectomorphism \cite{Del}. The polytopes appearing in this classifying scheme are thus called \emph{Delzant polytopes}.

Given a compact toric symplectic manifold $(M,\Omega, \mathbb{T}, \mu)$, Guillemin in \cite{Gul} showed that compatible $\mathbb{T}$-invariant K$\ddot{a}$hler structures are also determined by data specified on the moment polytope $\Delta$. The following is a sketch of the basic ideas.

Let $\mathring{\Delta}$ be the interior of $\Delta$. Then the open dense subset $\mathring{M}:=\mu^{-1}(\mathring{\Delta})$ consists of points at which $\mathbb{T}$ acts freely. Topologically, $\mu: \mathring{M}\rightarrow\mathring{\Delta}$ is a trivial principal $\mathbb{T}$-bundle over $\mathring{\Delta}$. Denote the set of $\mathbb{T}$-invariant complex structures on $M$ compatible with $\Omega$ by $K_\Omega^{\mathbb{T}}(M)$, i.e. the set of toric K$\ddot{a}$hler structures on $M$. Let $I\in K_\Omega^{\mathbb{T}}(M)$ and $\{X_j\}$ be the fundamental vector fields corresponding to a fixed basis $\{e_j\}$ of $\mathfrak{t}$. Then $\{X_j, IX_j\}$ is a global frame of $T\mathring{M}$ and the Lie bracket of any two vector fields in this frame vanishes. Let $\{\zeta_j,\vartheta_j\}$ be the dual frame on $T^*\mathring{M}$. Then $d\zeta_j=d\vartheta_j=0$ and thus locally $\zeta_j=d\theta_j$ and $\vartheta_j=du_j$. $\theta_j+\sqrt{-1}u_j$ are then local holomorphic coordinates of $\mathring{M}$ (these $u_j$'s are actually globally defined on $\mathring{M}$ due to the fact that $\mathring{\Delta}$ is simply connected). On the other side, $\{\vartheta_j\}$ and $\{d\mu_j\}$ determine the same integrable Lagrangian distribution $\mathcal{D}$ generated by those $X_j$'s and thus these $u_j$'s are functions depending only on $\mu$, i.e.
\begin{equation}du_j=-\sum_{k=1}^n\phi_{jk}(\mu)d\mu_k,\label{comp1}\end{equation}
or \footnote{As a convention, we have written $d\theta_j$'s or $d\mu_j$'s in a column. Similar notation is used below.}
\[I^*\left(
       \begin{array}{c}
         d\theta \\
         d\mu \\
       \end{array}
     \right)=\left(
               \begin{array}{cc}
                 0 & \phi \\
                 -\phi^{-1} & 0 \\
               \end{array}
             \right)\left(
                      \begin{array}{c}
                        d\theta \\
                        d\mu \\
                      \end{array}
                    \right).\]
These $\theta_j, \mu_j$ are actually Darboux coordinates, i.e. on $\mathring{M}$, $\Omega=\sum_{j=1}^nd\mu_j\wedge d\theta_j$.

Compatibility of $I$ with $\Omega$ forces the matrix $\phi=(\phi_{jk})$ to be symmetric and positive-definite, and integrability of Eq.~(\ref{comp1}) implies that $\phi$ ought to be the Hessian of a function $\tau$ defined on $\mathring{\Delta}$, or in other words $\tau$ is strictly convex. Due to the cental role of $\tau$, it is called the \emph{symplectic potential} of the invariant K$\ddot{a}$hler structure $I$, which provides a very useful computational tool in examining geometric ideas in K$\ddot{a}$hler geometry. The argument can go in the converse direction, i.e. a strictly convex function $\tau$ on $\mathring{\Delta}$ can be used to construct a toric K$\ddot{a}$hler structure on $\mathring{M}$. However, to extend the structure smoothly to the whole of $M$ requires $\tau$ to satisfy certain asymptotic conditions when approaching the boundary of $\Delta$.

Each Delzant polytope $\Delta$ can be associated with a canonical toric K$\ddot{a}$hler manifold $M_\Delta$ \cite{Del} and its symplectic potential can be totally determined by the data defining the polytope \cite{Gul}. If $\Delta$ in $\mathfrak{t}^*=(\mathbb{R}^n)^*$ is defined by
\[l_j(x):=(u_j, x)\geq \lambda_j,\quad j=1,2,\cdots, d\]
where the linear equations $l_j(x)=\lambda_j$ define faces of codimension 1 of $\Delta$ and $d$ is the number of such faces,
then the canonical symplectic potential on $M_\Delta$ is given by Guillemin's formula:
\begin{equation}\tau(x)=\frac{1}{2}\sum_{j=1}^dl_j(x)\ln l_j(x)\label{formu}\end{equation}

Boulanger's generalization in the GK setting went in a similar spirit. He considered $\mathbb{T}$-invariant GK structures $(\mathbb{J}_1, \mathbb{J}_2)$ of symplectic type on $(M,\Omega, \mathbb{T}, \mu)$, where $\mathbb{J}_2$ is a B-transform of $\mathbb{J}_\Omega$. Then the complex structure $I$ in the above argument is replaced by $J_+$ underlying the biHermitian description of the GK structure. However, the weaker condition of tameness no longer in general ensures that $\theta_j, \mu_j$ be Darboux coordinates. Boulanger thus focused on a special case to reserve this property. Denote the space of $\mathbb{T}$-invariant GK structures of symplectic type by $GK_\Omega^{\mathbb{T}}(M)$. Then an element of $GK_\Omega^{\mathbb{T}}(M)$ is called \emph{anti-diagonal} if for the underlying complex structures $J_\pm$ the condition $J_+\mathcal{D}=J_-\mathcal{D}$ holds, where $\mathcal{D}$ is again the Lagrangian distribution generated by $\{X_j\}$.

Let us introduce some notation before proceeding further. As in \cite{Bou}, denote this subset of anti-diagonal elements in $GK_\Omega^{\mathbb{T}}(M)$ by $DGK_\Omega^{\mathbb{T}}(M)$. Since an element in $GK_\Omega^{\mathbb{T}}(M)$ is completely parameterized by its associated complex structure $J_+$, we usually write $J_+\in GK_\Omega^{\mathbb{T}}(M)$ to imply this fact. Sometimes we also write $\mathbb{J}_1\in GK_\Omega^{\mathbb{T}}(M)$ if we want to emphasize the GC aspect of the underlying structures. Similar notation is adopted for elements in $DGK_\Omega^{\mathbb{T}}(M)$.

For $J_+\in DGK_\Omega^{\mathbb{T}}(M)$, $\theta_j, \mu_j$ are again Darboux coordinates (called \emph{admissible coordinates associated to $J_+$} in \cite{Bou}) and with such coordinates $J_\pm$ are of a form similar to Abreu-Guillemin's case
\begin{equation*}J_+^*\left(
       \begin{array}{c}
         d\theta \\
         d\mu \\
       \end{array}
     \right)=\left(
               \begin{array}{cc}
                 0 & \phi^T \\
                 -(\phi^{-1})^T & 0 \\
               \end{array}
             \right)\left(
                      \begin{array}{c}
                        d\theta \\
                        d\mu \\
                      \end{array}
                    \right),\quad J_-^*\left(
       \begin{array}{c}
         d\theta \\
         d\mu \\
       \end{array}
     \right)=\left(
               \begin{array}{cc}
                 0 & \phi \\
                 -\phi^{-1} & 0 \\
               \end{array}
             \right)\left(
                      \begin{array}{c}
                        d\theta \\
                        d\mu \\
                      \end{array}
                    \right)
\label{J+-}\end{equation*}
except that $\phi$ is not necessarily symmetric. Note that here $\phi^T$ denotes the transpose of $\phi$. Integrability of $J_\pm$ then forces the symmetric part $\phi_s$ ($=(\phi+\phi^T)/2$) of $\phi$ to be the Hessian of a function $\tau$ on $\mathring{\Delta}$ and the anti-symmetric part $C=\phi_a$ ($=(\phi-\phi^T)/2$) to be a constant $n\times n$ anti-symmetric matrix. Tameness then simply means that $\tau$ is strictly convex. A sketch of this argument can be found in the next subsection in a more general setting, or see \cite{Wang2} for a detailed account.

In \cite{Wang2}, it was further proved that Boulanger's $\tau$ is actually the symplectic potential of a \emph{genuine} toric K$\ddot{a}$hler structure $J_0$ canonically associated to $J_+$. Conversely, given a toric K$\ddot{a}$hler structure and an $n\times n$ constant anti-symmetric matrix $C$, there is a canonical way to construct an anti-diagonal toric GK structure of symplectic type. This is a rather nontrivial statement since it tells us that in this more general setting the symplectic potential $\tau$ has the same asymptotic behavior as that in the toric K$\ddot{a}$hler case when approaching the boundary of $\Delta$. Moreover, the underlying GC structure $\mathbb{J}_1$ is simply a B-transform of $\mathbb{J}_\beta$ induced from a $J_0$-holomorphic Poisson structure $\beta$ characterized by the matrix $C$, i.e. $\beta=1/2\sum_{j,k=1}^nC_{kj}X_j^h\wedge X_k^h$, where $X_j^h$ is the $J_0$-holomorphic part of $X_j$.

A fact we shall mention here is that by abuse of language, we will not distinguish $\mathbb{T}$-invariant smooth functions on $M$ (or $\mathring{M}$) from smooth functions on $\Delta$ (or $\mathring{\Delta}$) as is often done in the literature.
\subsection{General toric GK structures of symplectic type}\label{Log}
Let us begin with recalling a theorem from \cite{Wang2}.
Fix a basis $\{e_j\}$ of $\mathfrak{t}$ and let $\{\mu_j\}$ be the corresponding components of $\mu$. Note again that $\mathring{M}$ is a trivial principal $\mathbb{T}$-bundle over $\mathring{\Delta}$. Let $\zeta=\sum_j\zeta_je_j$ be a flat connection on this $\mathbb{T}$-bundle. Since the vertical distribution is Lagrangian, there exists a 1-form $\sigma_\zeta=\sum_j h_jd\mu_j$ with $h_j$ depending only on $\mu$ such that
\[\Omega=\sum_j d\mu_j\wedge \zeta_j+d\sigma_\zeta.\]
We call the matrix $F_\zeta:=(h_{k,j}-h_{j,k})$ the associated matrix of the connection $\zeta$. Obviously, $F_\zeta$ is determined by $\zeta$. If $F_\zeta$ happens to be a constant matrix, we say $\zeta$ is an \emph{admissible connection}. If furthermore $F_\zeta\equiv0$, we say $\zeta$ is of \emph{Darboux type}.

\begin{theorem}\label{GKG}(\cite{Wang2}) An element $J_+\in GK_\Omega^{\mathbb{T}}(\mathring{M})$ is determined by a triple $(\zeta^+,\tau, C)$ where $\zeta^+$ is an admissible connection on $\mathring{M}$, $C$ is an $n\times n$ constant anti-symmetric real matrix and $\tau$ is a strictly convex function on $\mathring{\Delta}$ such that its Hessian $\phi_s$ satisfies the condition below
\begin{equation}\phi_s+\frac{1}{4}F_{\zeta^+}(\phi_s)^{-1}F_{\zeta^+}\quad\textup{is positive-definite on $\mathring{\Delta}$}.\label{posi}\end{equation} Conversely, such a triple $(\zeta^+,\tau, C)$ also gives rise to an element in $GK_\Omega^{\mathbb{T}}(\mathring{M})$.\label{GK}
\end{theorem}
\begin{proof}For the reader's convenience, we sketch the proof here and a detailed version can be found in \cite{Wang2}.

Let $J_+\in GK_\Omega^{\mathbb{T}}(\mathring{M})$ and $X_j$ be the fundamental vector field generated by $e_j$. Tameness of $J_+$ with $\Omega$ assures that $\{X_j, J_+X_j\}$ be a global frame of $T\mathring{M}$. Let $\{\zeta_j^+, \vartheta_j\}$ be the corresponding dual frame of $T^*\mathring{M}$. Since $J_+$ is integrable and the action of $\mathbb{T}$ is abelian, $\zeta^+:=\sum_j\zeta_j^+e_j$ gives rise to a flat connection on $\mathring{M}$. Locally $\zeta_j^+=d\theta_j^+$, $\vartheta_j=du_j^+$ and $\{\theta_j^++\sqrt{-1}u_j^+\}$ is a local $J_+$-holomorphic coordinate system on $\mathring{M}$. Since $\{du_j^+\}$ and $\{d\mu_j\}$ determine the same distribution $\mathcal{D}$, $du_j^+=-\sum_k\phi_{jk}d\mu_k$,
where $\phi_{jk}$'s are functions only of $\mu$; in particular,
\begin{equation}J_+^*\left(
       \begin{array}{c}
         \zeta^+ \\
         d\mu \\
       \end{array}
     \right)=\left(
               \begin{array}{cc}
                 0 & \phi^T \\
                 -(\phi^{-1})^T & 0 \\
               \end{array}
             \right)\left(
                      \begin{array}{c}
                        \zeta^+ \\
                        d\mu \\
                      \end{array}
                    \right),\label{J+}\end{equation}
and for a certain matrix-valued function $F=(F_{kj})$,
\[\Omega=\sum_j d\mu_j\wedge \zeta^+_j+\frac{1}{2}\sum_{j,k}F_{kj}d\mu_j\wedge d\mu_k.\]

The same argument applies to $J_-$ as well. There should be a flat connection $\zeta^-$ and a matrix-valued function $\psi$ only of $\mu$ such that
\[J_-^*\left(
       \begin{array}{c}
         \zeta^- \\
         d\mu \\
       \end{array}
     \right)=\left(
               \begin{array}{cc}
                 0 & \psi^T \\
                 -(\psi^{-1})^T & 0 \\
               \end{array}
             \right)\left(
                      \begin{array}{c}
                        \zeta^- \\
                        d\mu \\
                      \end{array}
                    \right).\]
However $\psi$ is nothing else but $\phi^T$. Actually, in the coordinates $\{\theta_j^+, \mu_j\}$,
\begin{eqnarray*}\zeta^-&=&-\psi^TJ_-^*d\mu=\psi^T\Omega J_+\Omega^{-1}(d\mu)\\
&=&-\psi^T\Omega J_+(\partial_{\theta^+})=\psi^T\phi^{-1}\Omega(\partial_\mu)\\
&=&\psi^T\phi^{-1}(\zeta^++F d\mu).\end{eqnarray*}
Since $\zeta^+$ and $\zeta^-$ are both flat connections on the same principal $\mathbb{T}$-bundle, we must have
\[\zeta^-_j=\zeta^+_j+df_j\]
for some functions $f_j$ depending only on $\mu$. This observation implies $\psi^T\phi^{-1}=\textup{I}$ where $\textup{I}$ is the identity matrix or equivalently $\psi=\phi^T$ as required. Additionally, we must also have $F_{kj}=f_{j,k}$. Therefore, by taking a derivative, we have
\[F_{kj,l}=f_{j,kl}=f_{j,lk}=F_{lj,k},\]
which, together with $F_{kj}=-F_{jk}$, immediately implies that $F_{kj,l}=0$ and consequently that $F$ is actually an anti-symmetric constant matrix, i.e. $\zeta^\pm$ are both admissible connections.

Since $\zeta^+_j-\sqrt{-1}\sum_k\phi_{jk}d\mu_k$ and $\zeta^-_j-\sqrt{-1}\sum_k\phi_{kj}d\mu_k$ are holomorphic 1-forms w.r.t. $J_\pm$ respectively, integrability of $J_\pm$ thus implies
\begin{equation}\phi_{kj,l}=\phi_{lj,k}, \quad \phi_{jk,l}=\phi_{jl,k}.\label{pr}\end{equation}
Then we can conclude just as Boulanger had done in \cite{Bou} that the anti-symmetric part $\phi_a$ of $\phi$ should be a constant matrix $C$ and the symmetric part $\phi_s$ of $\phi$ be the Hessian of a function $\tau$ defined on $\mathring{\Delta}$.

To see what tameness of $J_+$ with $\Omega$ means, we should derive the matrix form of the metric $g$. Note that in the frame $\{\zeta^+, d\mu\}$,
 \[J_-^*\sim \left(
               \begin{array}{cc}
                 \textup{I} & -F \\
                 0 & \textup{I} \\
               \end{array}
             \right)\left(
                      \begin{array}{cc}
                        0 & \phi \\
                        -\phi^{-1} & 0 \\
                      \end{array}
                    \right)\left(
               \begin{array}{cc}
                 \textup{I} & F \\
                 0 & \textup{I} \\
               \end{array}
             \right)=\left(
                       \begin{array}{cc}
                         F\phi^{-1} & F\phi^{-1}F+\phi \\
                         -\phi^{-1} & -\phi^{-1}F \\
                       \end{array}
                     \right).
 \]
 Then from the formula $g=1/2(J_+^*+J_-^*)\Omega$, we can obtain the matrix form of $g$ relative to $\{\zeta^+, d\mu\}$:
 \[g\sim \left(
           \begin{array}{cc}
             (\phi^{-1})_s & \phi^{-1}F/2 \\
             -F(\phi^{T})^{-1}/2 & \phi_s\\
           \end{array}
         \right).
 \]
 It's elementary to find that positive-definiteness of $g$ is equivalent to that both $\phi_s$ and $\phi_s+1/4F(\phi_s)^{-1}F$ are positive-definite. Thus $\tau$ should satisfy the properties listed in the theorem. Clearly, the triple $(\zeta^+,\tau, C)$ determines $J_+$ uniquely.

 Conversely, given the triple $(\zeta^+,\tau, C)$ satisfying the conditions listed in the theorem, let $\phi_s$ be the Hessian of $\tau$ and $\phi=\phi_s+C$ and define $J_+$ in the manner of Eq.~(\ref{J+}). Obviously such a $J_+\in GK_\Omega^{\mathbb{T}}(\mathring{M})$.
\end{proof}

 Before moving on, let us motivate our further steps by giving some remarks on the implication of Thm.~\ref{GKG}. In this theorem, if $F=0$, then we recover Boulanger's result for anit-diagonal GK structures of symplectic type. In contrast with this more restrictive case, we should emphasize that in general \emph{two} constant anti-symmetric matrices $C$ and $F$ are involved in the characterization of a toric GK structure of symplectic type. Compared with $C$, this additional $F$ turns out to play a very different role: In the anti-diagonal case, only one flat connection $\zeta^+$ of Darboux type is involved, and furthermore in the \emph{single} frame $\{\zeta^+, d\mu\}$, $J_\pm$ can be anti-diagonalized \emph{simultaneously}. However, in the general case, \emph{three} flat connections are involved: two admissible connections $\zeta^\pm$ associated with $J_\pm$ respectively and a flat connection $\zeta$ of Darboux type, i.e., $\zeta:=(\zeta^++\zeta^-)/2$ such that $\Omega=\sum_jd\mu_j\wedge \zeta_j$. These connections are related by
\begin{equation}\zeta^\pm=\zeta\mp\frac{1}{2}F d\mu.\end{equation}
In particular, $J_\pm$ fails to be anti-diagonalized simultaneously in a single frame. To understand the roles played by $\tau, C$ and $F$, it turns out to be very important to distinguish among these flat connections.

As mentioned in the former subsection, the symplectic potential for $J_+\in DGK_\Omega^{\mathbb{T}}(M)$ is actually the symplectic potential of a genuine toric K$\ddot{a}$hler structure in Abreu-Guillemin theory. It's a natural question to ask whether the symplectic potential $\tau$ in Thm.~\ref{GK} for $J_+\in GK_\Omega^{\mathbb{T}}(M)$ comes from a genuine toric K$\ddot{a}$hler structure \emph{in general}. We shall provide an affirmative answer to this question, but in this subsection we only give a local and partial answer.

In the present context, for $J_+\in GK_\Omega^{\mathbb{T}}(M)$, due to Thm.~\ref{GK} we can define two \emph{new} complex structures $I_\pm$ on $\mathring{M}$ by claiming their matrix forms in the frame $\{\zeta, d\mu\}$ to be
\begin{equation}I_+^*\sim \left(
            \begin{array}{cc}
              0 & \phi^T \\
              -(\phi^{-1})^T & 0 \\
            \end{array}
          \right),\quad I_-^*\sim \left(
            \begin{array}{cc}
              0 & \phi \\
              -\phi^{-1} & 0 \\
            \end{array}
          \right).\label{nc}\end{equation}
It should be emphasized that though having the same matrix forms, $I_\pm$ are different from $J_\pm$ because they are defined using the flat connection $\zeta$ of Darboux type rather than the admissible ones $\zeta^\pm$ associated to $J_\pm$; in particular, up to now we only know that $I_\pm$ are defined on $\mathring{M}$ rather than $M$.

By construction, obviously we have $I_+\in DGK_{\Omega}^{\mathbb{T}}(\mathring{M})$, i.e. $I_+$ is an anti-diagonal toric GK structure of symplectic type on $\mathring{M}$. Then there is a \emph{fifth} complex structure $I_0$ (called the average complex structure of $I_\pm$ in \cite{Wang2}) whose matrix form w.r.t. $\{\zeta, d\mu\}$ is
\begin{equation}I_0^*\sim \left(
            \begin{array}{cc}
              0 & \phi_s \\
              -(\phi_s)^{-1} & 0 \\
            \end{array}
          \right).\label{aver}\end{equation}
Then we know from \cite[Thm.~4.4, 4.5]{Wang2} that $I_0\in K_{\Omega}^{\mathbb{T}}(\mathring{M})$ and $\tau$ is the symplectic potential of $I_0$ on $\mathring{M}$, and that $\phi_a=C$ determines an $I_0$-holomorphic Poisson structure $\beta$ on $\mathring{M}$.

There is a \emph{sixth almost} complex structure $J_0$ on $\mathring{M}$. Note that $\mathbb{J}_2$ is a B-transform of $\mathbb{J}_\Omega$ by the two form $b$. In this context, the classical infinitesimal action of $\mathfrak{t}$ on $M$ obtains a cotangent correction: $X_j\mapsto X_j-b(X_j)$. This action should be understood in the formalism of extended Lie algebra actions in \cite{BCG1} (or \cite{Wang} for this simple case). Note that
\[-\mathbb{J}_1(X_j-b(X_j)=\mathbb{J}_1\mathbb{J}_2^2(X_j-b(X_j))=\mathcal{G}\Omega(X_j)=-g^{-1}d\mu_j.\]
Let $Y_j:=-g^{-1}d\mu_j$. These $Y_j$'s are orthogonal to $X_k$'s. Actually,
\[g(Y_j, X_k)=-(d\mu_j, X_k)=0\]
because $\mu$ should be $\mathbb{T}$-invariant. Thus $\{X_j, Y_j\}$ does form a global frame of $T\mathring{M}$ and $J_0$ could be simply defined by letting $J_0X_j=Y_j$. In the frame $\{\partial_{\theta^+}, \partial_\mu\}$, the matrix form of $J_0$ is
\begin{equation}J_0\sim \left(
            \begin{array}{cc}
              -\frac{1}{2}\Xi^{-1}F(\phi_s)^{-1}\phi & -\Xi^{-1} \\
              \Xi[\textup{I}+(\frac{1}{2}\Xi^{-1}F(\phi_s)^{-1}\phi)^2] & \frac{1}{2}F(\phi_s)^{-1}\phi\Xi^{-1} \\
            \end{array}
          \right),\label{comp}
\end{equation}
where $\Xi=\phi_s+1/4F(\phi_s)^{-1}F$.

If $F=0$ (the three flat connections thus coincide), i.e. the toric GK structure $J_+$ is anti-diagonal, then the above $J_0$ is integrable, coincides with $I_0$ and plays a fundamental role in understanding the underlying geometry \cite{Wang2}. In our present setting, $J_0$ is not integrable, but since it is naturally associated to the GC structure $\mathbb{J}_1$ and may have some importance, we choose to include it here.

Another use of these $X_j, Y_j$ is that the smooth distribution $\mathcal{D}_1$ (in the sense of Sussmann \cite{Sus}) generated by them preserves $\beta_1$, as was noted in the remark of \cite[Prop.~4.6]{Wang}. This observation implies partially
\begin{proposition}
For $\mathbb{J}_1\in GK_\Omega^{\mathbb{T}}(M)$, points in $\mathring{M}$ are all regular, and the common type is the co-rank of the complex matrix $F/2-\sqrt{-1}\phi_a$.\label{reg}
\end{proposition}
\begin{proof}
Recall that the \emph{type} of $\mathbb{J}_1$ at a point $p\in\mathring{M}$ is the complex dimension transverse to the symplectic leaf of $\beta_1$ through $p$. $p$ is called \emph{regular} if this number is constant around $p$. Since the distribution $\mathcal{D}_1$ has full dimension on $\mathring{M}$, $\mathring{M}$ is actually a leaf of $\mathcal{D}_1$ of the highest dimension. Now that $\beta_1$ is preserved by $\mathcal{D}_1$, the rank of $\beta_1$ on $\mathring{M}$ has to be a constant, i.e., points in $\mathring{M}$ are all regular for $\mathbb{J}_1$.

Besides the above intrinsic proof of the first part of Prop.~\ref{reg}, we can give an alternative proof by a direct local computation. Note that $\beta_3=-1/2(J_+-J_-)\Omega^{-1}$. We can write down the matrix form of $\beta_3$ w.r.t. $\{\zeta^+, d\mu\}$:
\[\beta_3\sim \left(
                \begin{array}{cc}
                  -\phi_a & F\phi^{-1}/2 \\
                  (\phi^{-1})^TF/2 & -(\phi^{-1})_a \\
                \end{array}
              \right),
\]
or as a tensor, $\beta_3$ is
\begin{eqnarray*}
\left(
          \begin{array}{cc}
            \partial_{\theta^+}^T & (J_+\partial_{\theta^+})^T \\
          \end{array}
        \right)&\otimes&\left(
                 \begin{array}{cc}
                   \textup{I} & 0 \\
                   0 & -\phi^T \\
                 \end{array}
               \right)
\left(
                \begin{array}{cc}
                  -\phi_a & F\phi^{-1}/2 \\
                  (\phi^{-1})^TF/2 & -(\phi^{-1})_a \\
                \end{array}
              \right)\left(
                       \begin{array}{cc}
                         \textup{I} & 0 \\
                         0 & -\phi \\
                       \end{array}
                     \right)
              \left(
                       \begin{array}{c}
                         \partial_{\theta^+} \\
                         J_+\partial_{\theta^+} \\
                       \end{array}
                     \right)\\
                     &=&\left(
          \begin{array}{cc}
            \partial_{\theta^+}^T & (J_+\partial_{\theta^+})^T \\
          \end{array}
        \right)\otimes \left(
                         \begin{array}{cc}
                           -\phi_a & -F/2 \\
                           -F/2 & \phi_a \\
                         \end{array}
                       \right) \left(
                       \begin{array}{c}
                         \partial_{\theta^+} \\
                         J_+\partial_{\theta^+} \\
                       \end{array}
                     \right).
\end{eqnarray*}
Note that the type of $\mathbb{J}_1$ is actually half the real dimension of $\textup{ker}(J_+-J_-)$ and that the matrix $\left(
                         \begin{array}{cc}
                           -\phi_a & -F/2 \\
                           -F/2 & \phi_a \\
                         \end{array}
                       \right)$ is constant on $\mathring{M}$. We thus know that points in $\mathring{M}$ are all regular for $\mathbb{J}_1$; in particular, if we denote $z_i^+=\theta_i^++\sqrt{-1}u_i^+$, then it can be easily obtained that
                       \[\beta_+=2\sum_{i, j}[\frac{1}{2}F_{ij}-\sqrt{-1}(\phi_a)_{ij}]\partial_{z_j^+}\wedge \partial_{z_i^+}.\]
Consequently, the common type of $\mathbb{J}_1$ in $\mathring{M}$ is $n-\textup{rk}(F/2-\sqrt{-1}\phi_a)$.
\end{proof}
\emph{Remark}. Similarly, let $z_i^-=\theta_i^-+\sqrt{-1}u_i^-$. Then we have
\[\beta_-=2\sum_{i, j}[\frac{1}{2}F_{ij}-\sqrt{-1}(\phi_a)_{ij}]\partial_{z_j^-}\wedge \partial_{z_i^-}.\]
In particular, we find that $\mathbb{J}_1$ at a fixed point of the $\mathbb{T}$-action is of complex type because the vector fields $X_j$'s vanish there.
\section{Compactification}\label{Com}
In this section, we address the global smoothness of those structures defined on $\mathring{M}$ in \S~\ref{Log}, i.e. whether they can be extended smoothly on the whole of $M$. First we recall a basic lemma from \cite{Wang2}.
\begin{lemma}\label{c1}$I_+\in GK_{\Omega}^{\mathbb{T}}(\mathring{M})$ is the restriction of an element in $GK_{\Omega}^{\mathbb{T}}(M)$ on $\mathring{M}$ if and only if all the canonically associated tensors $\bar{g}$, $\bar{b}$ (see Eq.~\ref{sy}) and $(I_++I_-)^{-1}$ can be extended smoothly to $M$.
\end{lemma}
\begin{proof}For the reader's convenience, we give a sketch of the proof. Obviously, it suffices to prove the sufficiency part.

Due to the formulae $\bar{g}=-\frac{1}{2}\Omega(I_++I_-)$ and $\bar{b}=-\frac{1}{2}\Omega(I_+-I_-)$, if both $\bar{g}$, $\bar{b}$ can be extended smoothly to $M$, then $I_\pm$ are well-defined smooth tensors on $M$ for $\Omega$ is invertible. A continuity argument makes it clear that $I_\pm$ are actually integrable complex structures on $M$.

By continuity, $\bar{g}$ should be nonnegative-definite on $M\backslash \mathring{M}$. Since $\Omega=-2\bar{g}(I_++I_-)^{-1}$, the smoothness of $(I_++I_-)^{-1}$ implies that $\bar{g}$ must be non-degenerate on $M\backslash \mathring{M}$ and therefore positive-definite there.
\end{proof}
\emph{Remark}. Since $\bar{g}, \bar{b}$ are the symmetric and anti-symmetric parts of $\bar{g}+\bar{b}$ respectively, to establish that $\bar{g}, \bar{b}$ are actually smooth on $M$, it is enough to prove that the sum $\bar{g}+\bar{b}$ is smooth on $M$.

Now let us come back to the context of \S~\ref{Log} and let $I_+\in GK_\Omega^{\mathbb{T}}(\mathring{M})$ be defined in Eq.~(\ref{nc}). We can take $J_+$ as a reference element in $GK_\Omega^{\mathbb{T}}(M)$. If we can prove $\bar{g}-g$, $\bar{b}-b$ and $[(I_++I_-)/2]^{-1}-[(J_++J_-)/2]^{-1}$ all extend smoothly to $M$, then by Lemma~\ref{c1} $I_+$ is therefore the restriction of an element in $DGK_\Omega^{\mathbb{T}}(M)$.
\begin{lemma}Let $J_+\in GK_{\Omega}^{\mathbb{T}}(M)$ and $\Xi:=\phi_s+\frac{1}{4}F(\phi_s)^{-1}F$ in the context of Thm.~\ref{GK}. Then the inverse $\Xi^{-1}$ admits a smooth extension to $M$.\label{lemm}
\label{Xi}\end{lemma}
\begin{proof}In the frame $\{\partial_{\theta^+}, \partial_\mu\}$, the invertible map $(J_++J_-)/2$ has the matrix form
\[\frac{J_++J_-}{2}\sim \left(
                          \begin{array}{cc}
                            -(\phi^T)^{-1}F/2 & -(\phi^{-1})_s \\
                           \phi_s+F(\phi^T)^{-1}F/2 & F(\phi^T)^{-1}/2 \\
                          \end{array}
                        \right).\]
A more complicated yet elementary computation shows that its inverse has the matrix form
\begin{equation}(\frac{J_++J_-}{2})^{-1}\sim\left(
    \begin{array}{cc}
      \Xi^{-1}F(\phi_s)^{-1}\phi/2 & \Xi^{-1} \\
      (-\phi^T+AF\Xi^{-1}F/2)(\phi_s)^{-1}\phi & AF\Xi^{-1} \\
    \end{array}
  \right),\label{inver}
\end{equation}
where $A=\frac{\phi^T(\phi_s)^{-1}}{2}-\textup{I}$;
in particular, we find
\[\Omega((\frac{J_++J_-}{2})^{-1}\partial_{\theta^+_i}, \partial_{\theta^+_j})=\Omega(\sum_k(\Xi^{-1})^{ki}\partial_{\mu_k}, \partial_{\theta^+_j})=(\Xi^{-1})^{ji}.\]
Since $\Omega((\frac{J_++J_-}{2})^{-1}\partial_{\theta^+_i}, \partial_{\theta^+_j})$ is smooth on $M$, we know that $\Xi^{-1}$ admits a smooth extension to $M$.
\end{proof}
\begin{theorem}For $J_+\in GK_{\Omega}^{\mathbb{T}}(M)$, $I_+$ defined in Eq.~(\ref{nc}) is actually the restriction of an element in $DGK_\Omega^{\mathbb{T}}(M)$ on $\mathring{M}$.\label{M1}
\end{theorem}
\begin{proof}We take $J_+$ to be the reference toric GK structure. The first thing we shall do is to write down the matrix forms of $\bar{g}+\bar{b}$ and $[(I_++I_-)/2]^{-1}$ in the frame $\{\zeta^+, d\mu\}$ (in \S~\ref{Log}, $I_\pm$ are defined in the different frame $\{\zeta, d\mu\}$). Note that the two frames are related by
\[\left(
    \begin{array}{c}
      \zeta \\
      d\mu \\
    \end{array}
  \right)=\left(
            \begin{array}{cc}
              \textup{I} & F/2 \\
              0 & \textup{I} \\
            \end{array}
          \right)\left(
                   \begin{array}{c}
                     \zeta^+ \\
                     d\mu \\
                   \end{array}
                 \right)
\]
and that the matrix form of $\bar{g}+\bar{b}$ in the frame $\{\zeta, d\mu\}$ is
\[\bar{g}+\bar{b}\sim \left(
                                                \begin{array}{cc}
                                                  \phi^{-1} & 0 \\
                                                  0 & \phi \\
                                                \end{array}
                                              \right).\]
That's to say, as a tensor
\begin{eqnarray*}\bar{g}+\bar{b}&=&\left(
                                     \begin{array}{cc}
                                       \zeta^T & d\mu^T \\
                                     \end{array}
                                   \right)
\otimes \left(
                                                \begin{array}{cc}
                                                  \phi^{-1} & 0 \\
                                                  0 & \phi \\
                                                \end{array}
                                              \right)\left(
                                                       \begin{array}{c}
                                                         \zeta \\
                                                         d\mu \\
                                                       \end{array}
                                                     \right)
                                              \\
&=&\left(  \begin{array}{cc}
                                       \zeta^{+T} & d\mu^T \\
                                     \end{array}
                                   \right)\otimes\left(
                                                   \begin{array}{cc}
                                                     \textup{I} & 0 \\
                                                     -F/2 & \textup{I} \\
                                                   \end{array}
                                                 \right)\left(
                                                \begin{array}{cc}
                                                  \phi^{-1} & 0 \\
                                                  0 & \phi \\
                                                \end{array}
                                              \right)\left(
                                                   \begin{array}{cc}
                                                     \textup{I} & F/2 \\
                                                     0 & \textup{I} \\
                                                   \end{array}
                                                 \right)\left(
                                                       \begin{array}{c}
                                                         \zeta \\
                                                         d\mu \\
                                                       \end{array}
                                                     \right)\\
                                                     &=&\left(
                                     \begin{array}{cc}
                                       \zeta^{+T} & d\mu^T \\
                                     \end{array}
                                   \right)\otimes\left(
                                                   \begin{array}{cc}
                                                     \phi^{-1} & \phi^{-1}F/2 \\
                                                     -F\phi^{-1}/2 & \phi-F\phi^{-1}F/4 \\
                                                   \end{array}
                                                 \right)\left(
                                                       \begin{array}{c}
                                                         \zeta^+ \\
                                                         d\mu \\
                                                       \end{array}
                                                     \right).
                                   \end{eqnarray*}
Since $\zeta^+$ has no global meaning on $M$, we would like to replace it with $-\phi^TJ_+^*d\mu$ where $J_+^*d\mu$ is smooth on $M$. We obtain on $\mathring{M}$
\[\bar{g}+\bar{b}=\left(\begin{array}{cc}
                                       (J_+^*d\mu)^T & d\mu^T \\
                                     \end{array}
                                   \right)\otimes\left(
                                                   \begin{array}{cc}
                                                     \phi^{T} & -F/2 \\
                                                     F\phi^{-1}\phi^T/2 & \phi-F\phi^{-1}F/4 \\
                                                   \end{array}
                                                 \right)\left(
                                                       \begin{array}{c}
                                                        J_+^*d\mu\\
                                                        d\mu \\
                                                      \end{array}
                                                    \right).\]
Similarly, we have
\[g+b=\left(\begin{array}{cc}
                                       (J_+^*d\mu)^T & d\mu^T \\
                                     \end{array}
                                   \right)\otimes\left(
                                                   \begin{array}{cc}
                                                     \phi^T & -F \\
                                                     0 & \phi \\
                                                   \end{array}
                                                 \right)\left(
                                                       \begin{array}{c}
                                                        J_+^*d\mu\\
                                                        d\mu \\
                                                      \end{array}
                                                    \right).
                                   \]
Therefore, to see $\bar{g}+\bar{b}$ is smooth on $M$, we only need to check both $\phi^{-1}\phi^T$ and $\phi^{-1}$ can be extended smoothly to $M$. The conclusion for $\phi^{-1}$ is obvious because on $\mathring{M}$
 \[-\Omega(J_+\partial_{\theta_i^+}, \partial_{\theta_j^+})=\Omega(\sum_k(\phi^{-1})^{ki}\partial_{\mu_k}, \partial_{\theta_j^+})=(\phi^{-1})^{ji},\]
 and $\Omega(J_+\partial_{\theta_i^+}, \partial_{\theta_j^+})$ is smooth on $M$. Note that
 \[\phi^{-1}\phi^T=\phi^{-1}(\phi-2\phi_a)=\textup{I}-2\phi^{-1}\phi_a.\]
 So $\phi^{-1}\phi^T$ is as well smooth on $M$.

By Eq.~(\ref{inver}), in terms of $\{J_+^*d\mu, d\mu\}$ and $\{\partial_{\theta^+}, J_+\partial_{\theta^+}\}$, the tensor $[(J_++J_-)/2]^{-1}$ has the following form
\[\left(\begin{array}{cc}
                                       (J_+^*d\mu)^T & d\mu^T \\
                                     \end{array}
                                   \right)\otimes\left(
                          \begin{array}{cc}
                            -\phi\Xi^{-1}F(\phi_s)^{-1}\phi/2 & \phi\Xi^{-1}\phi \\
                            (-\phi^T+AF\Xi^{-1}F/2)(\phi_s)^{-1}\phi & -AF\Xi^{-1}\phi \\
                          \end{array}
                        \right)\left(
                                 \begin{array}{c}
                                   \partial_{\theta^+} \\
                                   J_+\partial_{\theta^+} \\
                                 \end{array}
                               \right)
                      \]
where $A=\phi^T(\phi_s)^{-1}/2-\textup{I}$. Similarly, the tensor $[(I_++I_-)/2]^{-1}$ has the following form:
 \[\left(\begin{array}{cc}
                                       (J_+^*d\mu)^T & d\mu^T \\
                                     \end{array}
                                   \right)\otimes\left(
                           \begin{array}{cc}
                             -\phi(\phi_s)^{-1}F/2 &\phi(\phi_s)^{-1}\phi  \\
                             -F(\phi_s)^{-1}F/4-\phi^T(\phi_s)^{-1}\phi & F(\phi_s)^{-1}\phi/2 \\
                           \end{array}
                         \right)\left(
                                 \begin{array}{c}
                                   \partial_{\theta^+} \\
                                   J_+\partial_{\theta^+} \\
                                 \end{array}
                               \right).
 \]
 Thus to prove that $[(I_++I_-)/2]^{-1}$ is globally well-defined on $M$, we must justify the following statements:\\
\textbf{\emph{ i}}) $\phi\Xi^{-1}F(\phi_s)^{-1}\phi-\phi(\phi_s)^{-1}F$ is smooth on $M$;\\
\textbf{\emph{ ii}}) $\phi(\Xi^{-1}-(\phi_s)^{-1})\phi$ is smooth on $M$;\\
\textbf{\emph{ iii}}) $AF\Xi^{-1}F(\phi_s)^{-1}\phi+1/2F(\phi_s)^{-1}F$ is smooth on $M$;\\
\textbf{\emph{ iv}}) $1/2F(\phi_s)^{-1}\phi+AF\Xi^{-1}\phi$ is smooth on $M$.

With Lemma~\ref{Xi} in mind, a careful analysis reveals that one only needs to check that the following matrix-valued functions
\[(\phi_s)^{-1},\quad \phi_s\Xi^{-1}\phi_s-\phi_s,\quad \Xi^{-1}\phi_s\]
are smooth on $M$. Note that
\begin{eqnarray*}\phi_s\Xi^{-1}\phi_s-\phi_s&=&[\Xi-1/4F(\phi_s)^{-1}F]\Xi^{-1}\phi_s-\phi_s\\
&=&\phi_s-1/4F(\phi_s)^{-1}F\Xi^{-1}\phi_s-\phi_s\\
&=&-1/4F(\phi_s)^{-1}F\Xi^{-1}\phi_s
\end{eqnarray*}
and
\[\Xi^{-1}\phi_s=\Xi^{-1}[\Xi-1/4F(\phi_s)^{-1}F]=\textup{I}-1/4\Xi^{-1}F(\phi_s)^{-1}F.\]
So to prove the theorem we only have to prove that $(\phi_s)^{-1}$ is smooth on $M$.\\
\emph{\textbf{Claim}}. $(\phi_s)^{-1}$ admits a smooth extension to $M$.
\begin{proof}
Let us compute the seemingly irrelevant quantity $((J_++J_-)^{-1}\partial_{\theta^+_i},J_+^*d\mu_j)$ first. From Eq.~(\ref{inver}), we have
\begin{eqnarray*}
2((J_++J_-)^{-1}\partial_{\theta_i^+},J_+^*d\mu_j)&=&-\sum_{k,l}[\Xi^{-1}F(\phi_s)^{-1}\phi]_{ki}(\phi^{-1})^{jl}(\partial_{\theta^+_k},\zeta^+_l)\\
&=&-[\Xi^{-1}F(\phi_s)^{-1}]_{ji}.
\end{eqnarray*}
Since $((J_++J_-)^{-1}\partial_{\theta_i^+},J_+^*d\mu_j)$ is a globally defined smooth function on $M$, we know that $\Xi^{-1}F(\phi_s)^{-1}$ is smooth on $M$. Additionally, we have
\[\textup{I}=\Xi^{-1}(\phi_s+1/4F(\phi_s)^{-1}F)=\Xi^{-1}\phi_s+1/4\Xi^{-1}F(\phi_s)^{-1}F\]
and consequently $\Xi^{-1}\phi_s$ is smoothly defined on $M$. Note that $\phi_s\geq \Xi$ on $\mathring{M}$ in the sense that their difference $-1/4F(\phi_s)^{-1}F$ is nonnegative-definite on $\mathring{M}$. Consequently we have the reversed inequality $(\phi_s)^{-1}\leq \Xi^{-1}$ on $\mathring{M}$, and
\[\det(\Xi^{-1}\phi_s)=\det{\Xi^{-1}}\times\det{\phi_s}\geq 1\]
on $\mathring{M}$. By continuity, $\det{(\Xi^{-1}\phi_s)}\geq 1$ on the whole of $M$, implying that $\Xi^{-1}\phi_s$ is both smooth and invertible on $M$. Therefore, $(\phi_s)^{-1}\Xi$ is also smooth on $M$. The factorization
\[(\phi_s)^{-1}=[(\phi_s)^{-1}\Xi]\times\Xi^{-1}\]
 then implies that $(\phi_s)^{-1}$ is also smooth on $M$.
\end{proof}
By Lemma~\ref{c1}, we thus have finally proved that $I_+\in DGK_\Omega^{\mathbb{T}}(M)$.
\end{proof}
\begin{corollary}\label{cano}The complex structure $I_0$ defined in Eq.~(\ref{aver}) admits a smooth extension on $M$. More precisely, $I_0\in K_\Omega^\mathbb{T}(M)$ and $\tau$ is actually the symplectic potential of this toric K$\ddot{a}$hler structure in Abreu-Guillemin theory.
\end{corollary}
\begin{proof}By Thm.~\ref{M1} $I_+$ is actually an anti-diagonal toric GK structure of symplectic type on $M$. The corollary then follows immediately from \cite[Thm.~4.9]{Wang2} because $I_0$ is the average complex structure of $I_+$ and $I_-$.
\end{proof}
\emph{Remark}. Due to the above results, we now have another convenient description of an element $J_+\in GK_\Omega^\mathbb{T}(M)$: It is characterized by the triple $(\tau, C, F)$ where $\tau$ is the symplectic potential of a toric K$\ddot{a}$hler structure $I_0$ and $C, F$ are two $n\times n$ constant anti-symmetric matrices such that the condition (\ref{posi}) is satisfied.
\begin{corollary}The almost complex structure $J_0$ defined in \S~\ref{Log} (see Eq.~(\ref{comp})) admits a smooth extension to $M$; in particular, $J_0$ is compatible with $\Omega$, i.e. $J_0$ is a toric almost K$\ddot{a}$hler structure.
\end{corollary}
\begin{proof}
The proof goes in the same spirit as the proof of Thm.~\ref{M1}. $J_0$ as a tensor has the following form:
\[\left(\begin{array}{cc}
                                       (J_+^*d\mu)^T & d\mu^T \\
                                     \end{array}
                                   \right)\otimes\left(
            \begin{array}{cc}
              \frac{1}{2}\phi\Xi^{-1}F(\phi_s)^{-1}\phi & -\phi\Xi^{-1}\phi \\
              \Xi[\textup{I}+(\frac{1}{2}\Xi^{-1}F(\phi_s)^{-1}\phi)^2] & -\frac{1}{2}F(\phi_s)^{-1}\phi\Xi^{-1}\phi \\
            \end{array}
          \right)\left(
                                 \begin{array}{c}
                                   \partial_{\theta^+} \\
                                   J_+\partial_{\theta^+} \\
                                 \end{array}
                               \right).\]
Similarly, $J_+$ as a tensor has the form:
\[\left(\begin{array}{cc}
                                       (J_+^*d\mu)^T & d\mu^T \\
                                     \end{array}
                                   \right)\otimes\left(
                                                   \begin{array}{cc}
                                                     0 & -\phi \\
                                                     \phi & 0 \\
                                                   \end{array}
                                                 \right)\left(
                                 \begin{array}{c}
                                   \partial_{\theta^+} \\
                                   J_+\partial_{\theta^+} \\
                                 \end{array}
                               \right).
                                   \]
Thus to see $J_0$ is globally defined on $M$, we have to prove the matrix-valued functions $\phi\Xi^{-1}F(\phi_s)^{-1}\phi$, $\phi\Xi^{-1}\phi-\phi$, $F(\phi_s)^{-1}\phi\Xi^{-1}\phi$ and $\Xi[\textup{I}+(\frac{1}{2}\Xi^{-1}F(\phi_s)^{-1}\phi)^2]-\phi$ can all be extended smoothly to $M$. In Lemma~\ref{lemm} and the proof of Thm.~\ref{M1}, we have already established the global smoothness of $\Xi^{-1}$, $\Xi$ $(\phi_s)^{-1}$ and $\Xi^{-1}\phi_s$, which is sufficient for establishing the global smoothness of $J_0$.

To see $J_0$ is compatible with $\Omega$, it's better to note that there is a natural $J_0$-Hermitian metric on $\mathring{M}$ defined by the GK structure on $M$. By identifying $T\mathring{M}$ with $\mathbb{K}\oplus \mathbb{J}_1\mathbb{K}$ where $\mathbb{K}$ is the subbundle of $T\mathring{M}\oplus T^*\mathring{M}$ generated by $X_j-b(X_j)$, the restriction of $-\mathbb{J}_1$ and the generalized metric $\mathcal{G}=-\mathbb{J}_1\mathbb{J}_2$ on $\mathbb{K}\oplus \mathbb{J}_1\mathbb{K}$ then gives rise to the almost complex structure $J_0$ and a Hermitian metric $\tilde{g}$ (then that $\tilde{g}$ is $J_0$-Hermitian is equivalent to that $\mathcal{G}$ is compatible with $\mathbb{J}_1$ ). One can easily check that on $\mathring{M}$, $\Omega=\tilde{g}J_0$. The detailed computation involved here is in essence the same as that appeared in the proof of \cite[Thm.~4.4]{Wang2} and thus omitted. By continuity, the conclusion can be finally established.
\end{proof}

 We have noted that for $\mathbb{J}_1\in GK_\Omega^{\mathbb{T}}(M)$, points in $\mathring{M}$ are all regular and on the other side fixed points are all of complex type. To conclude this section and also for completeness, let us have a very brief look at those points in $M\backslash\mathring{M}$.

 Let $P$ be an open face of codimension $k$ of $\Delta$, defined in $(\mathbb{R}^n)^*$ by
\[(u_{j_l}, \mu)=\lambda_{j_l},\quad l=1,2,\cdots, k,\]
and $V_P$ the linear subspace of $(\mathbb{R}^n)^*$ singled out by $(u_{j_l},\mu)=0$, $l=1,2,\cdots, k$. These $u_{j_l}\in \mathfrak{t}$ generate a subtorus $T_{0P}$ acting trivially on $M_P=\mu^{-1}(\bar{P})$, where $\bar{P}$ is the closure of $P$. Let $T_P$ be the quotient of $\mathbb{T}$ by $T_{0P}$. Note that intrinsically $C$ and $F$ are actually skew-symmetric bilinear functions on $\mathfrak{t}^*$ or in other words elements in $\wedge^2 \mathfrak{t}$. Let $c_P$, $f_P$ be the restriction of $c=1/2\sum_{j,k}C_{kj}e_j\wedge e_k$ and $f=1/2\sum_{j,k}F_{kj}e_j\wedge e_k$ on $V_P$ respectively.
\begin{theorem}Let $P$ be an open face of codimension $k$ of $\Delta$ as above. Then $M_P$ is a GK submanifold of $M$ for $\mathbb{J}_1\in GK_\Omega^{\mathbb{T}}(M)$. More precisely, its GK structure $(\mathbb{J}_{1P}, \mathbb{J}_{2P})$ belongs to $GK_{\Omega_P}^{\mathbb{T}_P}(M_P)$, where $\Omega_P$ is the restriction of $\Omega$ on $M_P$. $M_P$ inherits a toric K$\ddot{a}$hler structure from the canonical K$\ddot{a}$hler structure on $M$, which together with $c_P$ and $f_P$ characterizes the GK structure on $M_P$.
\end{theorem}
\begin{proof}Recall that $M_P$ is a GK submanifold of $M$ means that the pull-backs of the complex Dirac structures associated with $\mathbb{J}_1, \mathbb{J}_2$ to $M_P$ are themselves GC structures and form a GK structure on $M_P$. It is a rather standard argument to imply that $M_P$ is a complex submanifold relative to any one of the three complex structures $J_\pm$ and $I_0$. It is known that if a submanifold is both $J_+$- and $J_-$-invariant, then it is a GK submanifold (see for example \cite{Vai}). On the other side, the pull-back of $\mathbb{J}_2$ is of course of sympletic type with its symplectic form $\Omega_P$. These structures on $M_P$ are obviously $\mathbb{T}_P$-invariant and consequently the GK structure on $M_P$ lies in $GK_{\Omega_P}^{\mathbb{T}_P}(M_P)$.

Note that these matrices $C$, $F$ can be equivalently viewed as two canonical $I_0$-holomorphic Poisson structures on $M$ and $M_P$ is a Poisson submanifold relative to each of the two holomorphic Poisson structures. Obviously, the corresponding restricted holomorphic Poisson structures on $M_P$ are characterized by $c_P$ and $f_P$. To see these do characterize the toric GK structure on $M_P$, the most direct way is through the global formula (\ref{tran}) in \S~\ref{CON}, which shows how $J_+\in GK_\Omega^{\mathbb{T}}(M)$ arises from an element in $DGK_\Omega^{\mathbb{T}}(M)$. For the latter anti-diagonal case, the submanifold structure of $M_P$ has already been explored in \cite{Wang2}.
\end{proof}
\emph{Remark}. From the expression of $\beta_+$ we have derived in \S~\ref{Log}, we can find that on $\mu^{-1}(P)$ the type of $\mathbb{J}_1$ is $n-\textup{rk}(1/2f_P-\sqrt{-1}c_P)$ and the type of $\mathbb{J}_{1P}$ is $n-\textup{rk}(1/2f_P-\sqrt{-1}c_P)-k$.
\section{Symmetric toric GK structures of symplectic type}\label{symm}
 For $J_+\in GK_\Omega^{\mathbb{T}}(M)$, the way the underlying matrices $C$ and $F$ affect the GC structure $\mathbb{J}_1$ seems a bit complicated. In this section we specialize to the case where $C=0$ and $F\neq 0$. This is a case not touched at all in \cite{Bou}. We begin with an intrinsic characterization of this case. Recall that $\mathcal{D}$ is the distribution on $M$ generated by the infinitesimal action of $\mathfrak{t}$.
\begin{definition}If $J_+\in GK_\Omega^{\mathbb{T}}(M)$ satisfies the condition
\[(J_+-J_-)\mathcal{D}\subset \mathcal{D},\]
we call $J_+$ a symmetric toric GK structure of symplectic type on $M$.
\end{definition}
\begin{proposition}$J_+\in GK_\Omega^{\mathbb{T}}(M)$ is symmetric if and only if the underlying matrix $C$ in Thm.~\ref{GK} vanishes.
\end{proposition}
\begin{proof}Note that in the frame $\{\zeta^+, d\mu\}$,
\[\frac{J_+^*-J_-^*}{2}\sim \left(
                          \begin{array}{cc}
                            -F\phi^{-1}/2 & -F\phi^{-1}F/2-\phi_a \\
                            (\phi^{-1})_a & \phi^{-1}F/2 \\
                          \end{array}
                        \right).\]
Thus $(J_+-J_-)\mathcal{D}\subset \mathcal{D}$ if and only if $(\phi^{-1})_a\equiv0$. The latter is equivalent to that $\phi$ is symmetric, i.e. $C=0$.
\end{proof}
In the rest of this section, we always assume $J_+\in GK_\Omega^{\mathbb{T}}(M)$ is symmetric. In the frame $\{\zeta, d\mu\}$, the several geometric structures as linear maps are of the following matrix forms:
\[J_\pm^*\sim \left(
              \begin{array}{cc}
                \mp F\phi^{-1}/2 & \Xi \\
                -\phi^{-1} & \pm\phi^{-1}F/2 \\
              \end{array}
            \right),\quad g\sim \left(
                                  \begin{array}{cc}
                                    \phi^{-1} & 0 \\
                                    0 & \Xi \\
                                  \end{array}
                                \right),
\]
\[b\sim \left(
          \begin{array}{cc}
            0 & \phi^{-1}F/2 \\
            F\phi^{-1}/2 & 0 \\
          \end{array}
        \right),\quad \beta_1\sim \left(
                                    \begin{array}{cc}
                                      -F/2 & 0 \\
                                      0 & \phi^{-1}F\Xi^{-1}/2 \\
                                    \end{array}
                                  \right),
\]
where $\Xi=\phi+F\phi^{-1}F/4$. Note that points in $\mathring{M}$ are all regular for $\mathbb{J}_1$ and the common type is $n-\textup{rk}(F)$. In both frames $\{\zeta^+, d\mu\}$ and $\{\zeta^-, d\mu\}$, $\beta_3$ has the same matrix form
\begin{equation}\beta_3\sim\left(
    \begin{array}{cc}
      0 & F\phi^{-1}/2 \\
      \phi^{-1}F/2 & 0 \\
    \end{array}
  \right).
\label{beta}\end{equation}
That's to say,
\[\beta_3=-\frac{1}{2}\sum_{j,k}F_{kj}\partial_{\theta_j^+}\wedge \partial_{u_k^+}=-\frac{1}{2}\sum_{j,k}F_{kj}\partial_{\theta_j^-}\wedge \partial_{u_k^-},\]
where $z_j^\pm:=\theta_j^\pm+\sqrt{-1}u_j^\pm$ are $J_\pm$-holomorphic coordinates on $\mathring{M}$ respectively and consequently,
\[\beta_+=\sum_{j,k}F_{kj}\partial_{z_j^+}\wedge \partial_{z_k^+},\quad \beta_-=\sum_{j,k}F_{kj}\partial_{z_j^-}\wedge \partial_{z_k^-}.\]
 An astonishing fact which is crucial for understanding the underlying geometry is the following lemma. Note that $I_0$ is the toric K$\ddot{a}$hler structure canonically associated to $J_+$.
\begin{lemma}
$\beta_3$ is also the imaginary part of an $I_0$-holomorphic Poisson structure and $b$ is the imaginary part of an $I_0$-holomorphic 2-form.
\end{lemma}
\begin{proof}
Note that in the admissible coordinates $\theta_i, \mu_i$ the matrix form of $\beta_3$ is still of the form (\ref{beta}). Let $z_i=\theta_i+\sqrt{-1}u_i$ be $I_0$-holomorphic coordinates on $\mathring{M}$. Then
\begin{eqnarray*}\beta_3&=&\frac{1}{2}\sum_{j,k,l}F_{kj}(\phi^{-1})^{lk}\partial_{\theta_j}\wedge \partial_{\mu_l}=-\frac{1}{2}\sum_{j,k}F_{kj}\partial_{\theta_j}\wedge \partial_{u_k}\\
&=&-\frac{\sqrt{-1}}{2}\sum_{j,k}F_{kj}(\partial_{z_j}+\partial_{\bar{z}_j})\wedge (\partial_{z_k}-\partial_{\bar{z}_k})\\
&=&-\frac{\sqrt{-1}}{2}\sum_{j,k}F_{kj}\partial_{z_j}\wedge\partial_{z_k}+\frac{\sqrt{-1}}{2}\sum_{j,k}F_{kj}\partial_{\bar{z}_j}\wedge\partial_{\bar{z}_k}.
\end{eqnarray*}
This implies the conclusion for $\beta_3$ and that for $b$ can be obtained similarly.
\end{proof}
 If $S$ is an invertible endomorphism of $TM$, then $S$ acts naturally on the generalized tangent bundle $TM\oplus T^*M$ by acting only on the tangent part: $S\cdot (X+\xi)=S(X)+\xi$. We call this sort of transforms to be \emph{purely tangent}. Note that generally such a transform won't preserve the natural pairing on $TM\oplus T^*M$.
 \begin{theorem}If $J_+\in GK_\Omega^{\mathbb{T}}(M)$ is symmetric, then up to purely tangent transform, the underlying GC structure $\mathbb{J}_1$ is a B-transform of a GC structure $\mathbb{J}_\beta$ induced from an $I_0$-holomorphic Poisson structure $\beta=-\frac{1}{4}(I_0\beta_3+\sqrt{-1}\beta_3)$.\label{Symm}
\end{theorem}
\begin{proof}
Let $S=\frac{J_++J_-}{2}$. We can rewrite the matrix form of $\mathbb{J}_1$ in terms of $S, \beta_3$ and $b$. Actually,
\[\beta_1=-\frac{J_+-J_-}{2}g^{-1}=(\frac{J_++J_-}{2})^{-1}(\frac{J_+-J_-}{2})(\frac{J_++J_-}{2})g^{-1}=S^{-1}\beta_3,\]
where we have used the fact that
\[(J_++J_-)(J_+-J_-)=-(J_+-J_-)(J_++J_-).\]
Similarly,
\begin{eqnarray*}-\frac{1}{2}(\omega_+-\omega_-)&=&-g(\frac{J_+-J_-}{2})=gS^{-1}(\frac{J_+-J_-}{2})S\\
&=&-\Omega S S^{-1}(\frac{J_+-J_-}{2})S\\
&=&bS.\end{eqnarray*}
Therefore,
\[\mathbb{J}_1=\left(
                 \begin{array}{cc}
                   -S & S^{-1}\beta_3 \\
                   bS & S^* \\
                 \end{array}
               \right)=\left(
                         \begin{array}{cc}
                           S^{-1} & 0 \\
                           0 & \textup{Id} \\
                         \end{array}
                       \right)\left(
                                \begin{array}{cc}
                                  -S & \beta_3 \\
                                  b & S^* \\
                                \end{array}
                              \right)\left(
                                       \begin{array}{cc}
                                         S & 0 \\
                                         0 & \textup{Id} \\
                                       \end{array}
                                     \right).
\]
Now we shall prove there is a 2-form $b_1$ on $M$ such that
\[\left(
                                \begin{array}{cc}
                                  -S & \beta_3 \\
                                  b & S^* \\
                                \end{array}
                              \right)=\left(
                                        \begin{array}{cc}
                                          \textup{Id} & 0 \\
                                          -b_1 &\textup{Id} \\
                                        \end{array}
                                      \right)\left(
                                               \begin{array}{cc}
                                                 -I_0 & \beta_3 \\
                                                 0 & I_0^* \\
                                               \end{array}
                                             \right)\left(
                                        \begin{array}{cc}
                                          \textup{Id} & 0 \\
                                          b_1 &\textup{Id} \\
                                        \end{array}
                                      \right).
                              \]
$b_1$ should satisfy the following two equations:
\[S=I_0-\beta_3b_1,\quad b=b_1I_0-b_1\beta_3b_1+I_0^*b_1.\]
We can choose $b_1=-\frac{1}{4}F_{kj}d\mu_j\wedge d\mu_k$. By using those matrix forms involved, one can easily check that this $b_1$ really fulfills the above two equations on $\mathring{M}$. Since $b_1$ is a global 2-form on $M$, a continuity argument then completes the proof.
\end{proof}
\emph{Remark}. It seems that the \emph{novel} way $\mathbb{J}_1$ arises in the above theorem from a holomorphic Poisson structure did not appear before in the literature. It may have some interest to look close at such structures and we shall do this elsewhere.

The following proposition may have some relevance in understanding the implication of the almost complex structure $J_0$.
\begin{proposition}The almost complex structure $J_0$ defined in \S~\ref{Log}, the complex structure $I_0$ and $\beta_1$ are compatible in the sense that $J_0\beta_1=\beta_1I_0^*$.
\end{proposition}
\begin{proof}Note that in the frame $\{\zeta, d\mu\}$, it's elementary to find that the matrix form of $J_0$ is
\[J_0\sim \left(
            \begin{array}{cc}
              0 & -\Xi^{-1} \\
              \Xi & 0 \\
            \end{array}
          \right).
\]Then the result can be obtained by directly using the matrix forms of $I_0$ and $\beta_1$.
\end{proof}

\section{Constructing toric GK structures from toric K$\ddot{a}$hler structures}\label{CON}
In \S~\ref{Com}, we have proved that underlying a toric GK structure $J_+$ of symplectic type on $M$, there is a genuine toric K$\ddot{a}$hler structure canonically associated to it. Thus Thm.~\ref{GK} suggests the possibility that we could construct a nontrivial toric GK structure when a toric K$\ddot{a}$hler structure and two $n\times n$ anti-symmetric constant matrices $C$ and $F$ are given (a basis $\{e_j\}$ of $\mathfrak{t}$ is fixed). For the anti-diagonal case ($F=0$), this was established in \cite{Wang2} \emph{without any further restriction on $C$}. The goal of this section is basically to extend this result in its full generality.

The first thing one should bear in mind is that in our present setting, to establish a similar result, \emph{a priori} the matrix $F$ cannot be arbitrary since the condition (\ref{posi}) is \emph{really} a restriction on the magnitude of $F$. This is another fundamental distinction between the roles of $C$ and $F$. Let us explain this in some detail. If $\phi_s$ is the Hessian of the symplectic potential $\tau$ of $I_0\in K_\Omega^{\mathbb{T}}(M)$, then since $(\phi_s)^{-1}$ is nonnegative-definite on $M$, we can take its square root $(\phi_s)^{-1/2}$, which is continuous on $M$ \cite{Chen}. For later convenience, we replace the condition (\ref{posi}) by
\begin{equation}
\textup{I}+\frac{1}{4}[(\phi_s)^{-1/2}F(\phi_s)^{-1/2}]^2 \quad \textup{is positive-defintie on $\Delta$}.\label{posi1}
\end{equation}
Note that on $\mathring{M}$,
\[\textup{I}+\frac{1}{4}[(\phi_s)^{-1/2}F(\phi_s)^{-1/2}]^2=(\phi_s)^{-1/2}\Xi(\phi_s)^{-1/2}.\]Thus if $J_+\in GK_\Omega^{\mathbb{T}}(M)$ and $F$ is the underlying matrix in Thm.~\ref{GK}, then from the proof of Thm.~\ref{M1} we know that $F$ should satisfy the condition (\ref{posi1}). Conversely, if $I_0\in K_\Omega^{\mathbb{T}}(M)$ and $\phi_s$ is the Hessian of its symplectic potential $\tau$, then (\ref{posi1}) surely implies (\ref{posi})--we conjecture these two conditions are actually equivalent in this setting, but up to now we only know this does hold when $n=2$ (see Example~\ref{ex1}).
\begin{proposition}
Let $\mathcal{A}_n$ be the linear space of $n\times n$ anti-symmetric real matrices $F$ equipped with the norm $\|F\|:=\sqrt{-\textup{tr}(F^2)}$ and $\phi_s$ the Hessian of the symplectic potential $\tau$ of $I_0\in K_\Omega^{\mathbb{T}}(M)$. Then the subset $\mathcal{A}_\tau$ of $F\in \mathcal{A}_n$ satisfying the condition (\ref{posi1}) is a bounded open convex cone in $\mathcal{A}_n$.
\end{proposition}
\begin{proof}Fix a point $x\in \Delta$ and let $F_x:=(\phi_s)^{-1/2}(x)F(\phi_s)^{-1/2}(x)$. Then $F\in\mathcal{A}_\tau$ implies that $-\frac{1}{4}F_x^2<\textup{I}$ and consequently
\[\|F_x\|^2=-\textup{tr}(F_x^2)<4n.\]
This shows that $\mathcal{A}_\tau$ is bounded.

For $F\in\mathcal{A}_\tau$, since (\ref{posi1}) is an open condition, for each $x_0\in \Delta$, there is a neighbourhood $U_F^{x_0}\subset\mathcal{A}_n$ of $F$ and a neighborhood $V_{x_0}\subset \Delta$ of $x_0$ such that
\[-\frac{1}{4}F_x^2<\textup{I},\quad \forall F\in U_F^{x_0},\quad x\in V_{x_0}.\]
Now that $\Delta$ is compact, there is a finite subset $\{x_i\}\subset \Delta$ such that $\Delta=\cup_i V_{x_i}$. Then
$\cap_iU_F^{x_i}\subset \mathcal{A}_\tau$
and is an open neighbourhood of $F$ in $\mathcal{A}_n$. Thus $\mathcal{A}_\tau$ is open in $\mathcal{A}_n$.

Obviously, $0\in \mathcal{A}_\tau$. If $F\in \mathcal{A}_\tau$, then the line $(1-t)\times 0+tF, t\in[0,1]$ also lies in $\mathcal{A}_\tau$. Thus $\mathcal{A}_\tau$ is a cone with $0$ as its vertex.

To see $\mathcal{A}_\tau$ is convex, let $F_1, F_2\in \mathcal{A}_\tau$ and $\lambda\in(0, 1)$. We should prove $F_\lambda:=\lambda F_1+(1-\lambda)F_2\in \mathcal{A}_\tau$. It suffices to prove $-F_{\lambda x}^2<4\times\textup{I}$ for arbitrary $x\in \Delta$. Note that $F_{\lambda x}=\lambda F_{1x}+(1-\lambda)F_{2x}$ and let $|\cdot|$ denote the usual Euclidean norm on $\mathbb{R}^n$. For $0\neq v\in \mathbb{R}^n$, we have
\begin{eqnarray*}
(F_{\lambda x}v, F_{\lambda x}v)&=&\lambda^2(F_{1x}v, F_{1x}v)+(1-\lambda)^2(F_{2x}v, F_{2x}v)+2\lambda(1-\lambda)(F_{1x}v,F_{2x}v)\\
&\leq& \lambda^2|F_{1x}v|^2+(1-\lambda)^2|F_{2x}v|^2+2\lambda(1-\lambda)|F_{1x}v||F_{2x}v|\\
&=&(\lambda |F_{1x}v|+(1-\lambda)|F_{2x}v|)^2\\
&<& [\lambda\times 2|v|+(1-\lambda)\times 2|v|]^2\\
&=&4|v|^2,
\end{eqnarray*}
which establishes what we want. Note here the fourth line uses the fact that $F_1, F_2\in \mathcal{A}_\tau$.
\end{proof}
\emph{Remark}. For each Delzant polytope $\Delta$, let $M_\Delta$ be the standard toric K$\ddot{a}$hler manifold associated with $\Delta$. In this case $\mathcal{A}_\tau$ is completely determined by $\Delta$ itself as well as the symplectic potential $\tau$ is.

Let $I_0\in K_\Omega^{\mathbb{T}}(M)$, $\zeta$ be the flat connection of Darboux type naturally associated to $I_0$, and $\phi_s$ the Hessian of the symplectic potential $\tau$. If $C$, $F$ are two $n\times n$ anti-symmetric constant matrices and $F\in \mathcal{A}_\tau$, then as what Thm.~\ref{GK} tells us, we can define two one-parameter families of admissible connections $\zeta^{\pm t}=\zeta\mp\frac{t}{2} Fd\mu$ where $t\in [0, 1]$ and define $J_+^t\in GK_\Omega^{\mathbb{T}}(\mathring{M})$ by
\[J_+^{t*}\left(
       \begin{array}{c}
         \zeta^{+t} \\
         d\mu \\
       \end{array}
     \right)=\left(
               \begin{array}{cc}
                 0 & \phi^T \\
                 -(\phi^{-1})^T & 0 \\
               \end{array}
             \right)\left(
                      \begin{array}{c}
                        \zeta^{+t} \\
                        d\mu \\
                      \end{array}
                    \right),\]
where $\phi=\phi_s+C$. In this context,
\[J_-^{t*}\left(
       \begin{array}{c}
         \zeta^{-t} \\
         d\mu \\
       \end{array}
     \right):=-\Omega J_+^t\Omega^{-1}\left(
       \begin{array}{c}
         \zeta^{-t} \\
         d\mu \\
       \end{array}
     \right)=\left(
               \begin{array}{cc}
                 0 & \phi \\
                 -\phi^{-1} & 0 \\
               \end{array}
             \right)\left(
                      \begin{array}{c}
                        \zeta^{-t} \\
                        d\mu \\
                      \end{array}
                    \right).\]
\begin{theorem}Let $I_0\in K_\Omega^{\mathbb{T}}(M)$, $C, F$ be two $n\times n$ anti-symmetric constant matrices such that $F\in \mathcal{A}_\tau$, and $J_+^t\in GK_\Omega^{\mathbb{T}}(\mathring{M})$ defined as above. Then $J_+^t\in GK_\Omega^{\mathbb{T}}(M)$ for each $t\in [0, 1]$.\label{cons}
\end{theorem}
\begin{proof}To see $J_+^t$ is smooth on $M$, we can resort to another global description of $J_+^t$ and its symplectic adjoint $J_-^t$. First we can define another complex structure $I_+$ as follows:
\[I_+^*\left(
       \begin{array}{c}
         \zeta \\
         d\mu \\
       \end{array}
     \right)=\left(
               \begin{array}{cc}
                 0 & \phi^T \\
                 -(\phi^T)^{-1} & 0 \\
               \end{array}
             \right)\left(
                      \begin{array}{c}
                        \zeta \\
                        d\mu \\
                      \end{array}
                    \right).\]
Then due to \cite[Thm.~4.11]{Wang2}, $I_+$ is globally well-defined on $M$ and in particular $I_+\in DGK_\Omega^{\mathbb{T}}(M)$.

Define a map $\mathcal{F}_t: TM\rightarrow TM$ by
\[\mathcal{F}_t=\textup{Id}-\frac{t}{2}\Omega^{-1}\hat{F},\]
where $\hat{F}=1/2\sum_{j,k}F_{kj}d\mu_j\wedge d\mu_k$. This map $\mathcal{F}_t$ is smoothly well-defined on $M$. Then we have
\[\mathcal{F}_t^*\left(
                   \begin{array}{c}
                     \zeta \\
                     d\mu \\
                   \end{array}
                 \right)=\left(
                           \begin{array}{cc}
                             \textup{I} & -\frac{t}{2}F \\
                             0 & \textup{I} \\
                           \end{array}
                         \right)\left(
                   \begin{array}{c}
                     \zeta \\
                     d\mu \\
                   \end{array}
                 \right)=\left(
       \begin{array}{c}
         \zeta^{+t} \\
         d\mu \\
       \end{array}
     \right)
\]
and consequently
\begin{equation}J_+^{t*}=\mathcal{F}_t^*I_+^*(\mathcal{F}_t^*)^{-1},\quad J_-^{t*}=(\mathcal{F}_t^*)^{-1}I_-^*\mathcal{F}_t^*\label{tran}\end{equation}
where $I_-$ is the symplectic adjoint of $I_+$. This shows that $J_\pm^t$ are both smooth on $M$.

To see $J_+^t\in GK_\Omega^{\mathbb{T}}(M)$ for each $t\in[0, 1]$, by Lemma.~\ref{c1}, we only need to prove the global smoothness of $(J_+^t+J_-^t)^{-1}$. We shall adopt a similar strategy to that of the proof of Thm.~\ref{M1}. This time we choose the toric K$\ddot{a}$hler structure $I_0$ as the reference.

Let $\theta_j, \mu_j$ be the admissible coordinates associated to $I_0$. Then in the frame $\{\partial_\theta, \partial_\mu\}$, $(J_+^t+J_-^t)/2$ has the following matrix form:
\[\frac{J_+^t+J_-^t}{2}\sim \left(
                          \begin{array}{cc}
                            \frac{t}{2}(\phi^{-1})_aF & -(\phi^{-1})_s \\
                          \phi_s+\frac{t^2}{4}F(\phi^{-1})_sF & -\frac{t}{2}F(\phi^{-1})_a \\
                          \end{array}
                        \right),\]
and consequently (see Eq.~(\ref{inver}))
\begin{eqnarray*}(\frac{J_+^t+J_-^t}{2})^{-1}&\sim &\left(
                           \begin{array}{cc}
                             \textup{I} & 0 \\
                             \frac{tF}{2} & \textup{I} \\
                           \end{array}
                         \right)\left(
    \begin{array}{cc}
      \frac{t}{2}\Xi_t^{-1}F(\phi_s)^{-1}\phi & \Xi_t^{-1} \\
      (-\phi^T+\frac{t^2}{2}AF\Xi_t^{-1}F)(\phi_s)^{-1}\phi & tAF\Xi_t^{-1} \\
    \end{array}
  \right)\left(
                           \begin{array}{cc}
                             \textup{I} & 0 \\
                             -\frac{tF}{2} & \textup{I} \\
                           \end{array}
                         \right)\\
  &=&\left(
       \begin{array}{cc}
         \frac{t}{2}\Xi_t^{-1}FB & \Xi_t^{-1} \\
         -\phi^T(\phi_s)^{-1}\phi+\frac{t^2}{4}B^TF\Xi_t^{-1}FB & \frac{t}{2}B^TF\Xi_t^{-1} \\
       \end{array}
     \right),
                          \end{eqnarray*}
                          where $\Xi_t=\phi_s+\frac{t^2}{4}F(\phi_s)^{-1}F$, $A=\frac{\phi^T(\phi_s)^{-1}}{2}-\textup{I}$, $B=(\phi_s)^{-1}\phi_a$ and $\phi_a=C$.
Then as a tensor $(\frac{J_+^t+J_-^t}{2})^{-1}$ is of the following form:
\[\left(
  \begin{array}{cc}
    (I_0^*d\mu)^T & d\mu^T \\
  \end{array}
\right)\otimes\left(
       \begin{array}{cc}
         -\frac{t}{2}\phi_s\Xi_t^{-1}FB & \phi_s\Xi_t^{-1}\phi_s \\
         -\phi^T(\phi_s)^{-1}\phi+\frac{t^2}{4}B^TF\Xi_t^{-1}FB & -\frac{t}{2}B^TF\Xi_t^{-1}\phi_s \\
       \end{array}
     \right)\left(
              \begin{array}{c}
                \partial_\theta \\
                I_0\partial_\theta \\
              \end{array}
            \right).\]
Similarly, the tensor $(\frac{I_0+I_0^\Omega}{2})^{-1}=I_0^{-1}=-I_0$ is of the form:
\[\left(
  \begin{array}{cc}
    (I_0^*d\mu)^T & d\mu^T \\
  \end{array}
\right)\otimes\left(
                \begin{array}{cc}
                  0 & \phi_s \\
                  -\phi_s & 0 \\
                \end{array}
              \right)\left(
              \begin{array}{c}
                \partial_\theta \\
                I_0\partial_\theta \\
              \end{array}
            \right).
\]

Therefore, to prove the global smoothness of $(\frac{J_+^t+J_-^t}{2})^{-1}$, we have to check the following statements:\\
\textbf{\emph{i}}) $\phi_s\Xi_t^{-1}F(\phi_s)^{-1}\phi_a$ is smooth on $M$;\\
\textbf{\emph{ii}}) $\phi_s\Xi_t^{-1}\phi_s-\phi_s$ is smooth on $M$;\\
\textbf{\emph{iii}}) $ -\phi^T(\phi_s)^{-1}\phi-\frac{t^2}{4}\phi_a(\phi_s)^{-1}F\Xi_t^{-1}F(\phi_s)^{-1}\phi_a+\phi_s$ is smooth on $M$.\\
With the fact that $(\phi_s)^{-1}$ is smooth on $M$ in mind, a careful but elementary analysis shows that we only need to check that the matrix-valued function $\Xi_t^{-1}\phi_s$ is smooth on $M$. Note that
\[(\phi_s)^{-1}\Xi_t=\textup{I}+\frac{t^2}{4}[(\phi_s)^{-1}F]^2\]
is smooth on $M$. Thus to complete the proof of the theorem, it suffices to prove that $(\phi_s)^{-1}\Xi_t$ is also non-degenerate on $M\backslash\mathring{M}$. Note that on $M$ we have
\begin{eqnarray*}\det((\phi_s)^{-1}\Xi_t)&=&\det[(\phi_s)^{-1/2}\Xi_t(\phi_s)^{-1/2}]=\det[\textup{I}+\frac{t^2}{4}((\phi_s)^{-1/2}F(\phi_s)^{-1/2})^2]\\
&\geq&\det[\textup{I}+\frac{1}{4}((\phi_s)^{-1/2}F(\phi_s)^{-1/2})^2]>0, \end{eqnarray*}
where the condition (\ref{posi1}) is used. This completes the proof.
\end{proof}
\emph{Remark}. i) The map $\mathcal{F}_t$ in the proof can also be adapted to give another simpler proof of the global smoothness of $I_\pm$ in Thm.~\ref{M1}. ii) Note that at a fixed point of the $\mathbb{T}$-action, $d\mu=0$ and consequently the map $\mathcal{F}_t$ is the identity map there. Thus at a fixed point, the type of $\mathbb{J}_1^t$ associated to $J_+^t$ is always the same as that of the GC structure associated to $I_+$. Since $I_+\in  DGK_{\Omega}^{\mathbb{T}}(M)$, $\mathbb{J}_1^t$ should be of complex type at those fixed points due to the theory developed in \cite{Wang2}, just as was observed before from other viewpoints.
\begin{example}\label{ex1}Let us analyse the case $n=2$ in some detail. If
\[\phi_s=\left(
           \begin{array}{cc}
             \tau_{11} & \tau_{12} \\
             \tau_{12} & \tau_{22} \\
           \end{array}
         \right),\quad F=\left(
                           \begin{array}{cc}
                             0 & f \\
                             -f & 0 \\
                           \end{array}
                         \right),\quad C=\left(
                                         \begin{array}{cc}
                                           0 & c \\
                                           -c & 0 \\
                                         \end{array}
                                       \right),
\]
where $\phi_s$ is the Hessian of the symplectic potential $\tau$ of $I_0\in K_\Omega^{\mathbb{T}}(M)$, then the condition (\ref{posi}) amounts to that
\[\Xi=\phi_s+\frac{1}{4}F(\phi_s)^{-1}F=\phi_s(1-\frac{f^2}{4\det \phi_s})\]
is positive-definite on $\mathring{M}$ or equivalently $1-\frac{f^2}{4\det \phi_s}>0$ on $\mathring{M}$. From the Abreu-Guillemin theory, we know that $1/\det \phi_s$ is smooth on $M$ and in particular vanishes on $M\backslash\mathring{M}$. Therefore, the condition (\ref{posi}) actually implies that $1-\frac{f^2}{4\det \phi_s}>0$ holds on the whole of $M$ and thus
\[(\phi_s)^{-1}\Xi=(1-\frac{f^2}{4\det \phi_s})\times \textup{I}\]
is non-degenerate on $M$. Then in this dimension, the seemingly weaker condition (\ref{posi}) is enough for the validity of the conclusion of Thm.~\ref{cons}.

Let $m$ be the maximum of $1/\det \phi_s$ on $\Delta$. Then the condition $1-\frac{f^2}{4\det \phi_s}>0$ is equivalent to \[f\in (-\frac{2}{\sqrt{m}}, \frac{2}{\sqrt{m}})\cong\mathcal{A}_\tau.\]

If $c^2+f^2\neq 0$, then we can find the 2-form $b=-1/2\Omega(J_+-J_-)$ satisfies
\[-\det\phi\times b^2/2=(c^2+f^2/4)d\theta_1d\theta_2d\mu_1d\mu_2.\]
Consequently, points in $\mathring{M}$ are all regular for $\mathbb{J}_1$, in perfect agreement with the general result of Prop.~\ref{reg}.
\end{example}

 It is known that on a compact toric symplectic manifold $(M,\Omega, \mathbb{T}, \mu)$, the space of compatible toric K$\ddot{a}$hler structures is, modulo the action of $\mathbb{T}$-equivariant symplecomorphisms, a space $\mathrm{K}$ of continuous functions $\tau$ (symplectic potentials) on the moment polytope $\Delta$ satisfying the following two conditions \cite[Prop.~5]{Ap}:\\
i) The restriction of $\tau$ to any open face of $\Delta$ is a smooth strictly convex function;\\
ii) $\tau-\tau_0$ is smooth on $\Delta$, where $\tau_0$ is the function given by (\ref{formu}).

According to our results up to now, we have the following theorem.
\begin{theorem}For a given compact toric symplectic manifold $(M,\Omega, \mathbb{T}, \mu)$, the space $\mathcal{GK}$ of toric GK structures of symplectic type modulo the action of $\mathbb{T}$-equivariant symplectomorphisms, is the set
\[\mathcal{GK}_\Delta:=\{(\tau, F, C)|\tau\in \mathrm{K}, F\in \mathcal{A}_\tau, C\in \mathcal{A}_n\}.\]
\end{theorem}
\begin{proof}Note that from our previous results, the underlying matrices $F, C$ associated to an element $J_+\in GK_\Omega^{\mathbb{T}}(M)$ can be viewed as obtained from two holomorphic Poisson structures w.r.t. the canonically associated K$\ddot{a}$hler structure $I_0$. If two elements in $GK_\Omega^\mathbb{T}(M)$ have the same triple $(\tau, F, C)$ to characterize them, then the canonical underlying K$\ddot{a}$hler structures $I_0, I_0'$ are related by a $\mathbb{T}$-equivariant symplectomorphism $\Phi$. $\Phi$, as a holomorphic isomorphism between $(M, I_0)$ and $(M, I_0')$, also transforms canonically the $I_0$-holomorphic Poisson structures associated to $F, C$ to the $I_0'$-holomorphic Poisson structures associated to $F, C$.
\end{proof}
For later use, let us collect the matrix forms of several geometric structures associated to $J_+\in GK_\Omega^{\mathbb{T}}(M)$ in the frame $\{\zeta, d\mu\}$ (we only consider $t=1$ in the construction of Thm.~\ref{cons}):
\[J_+\sim \left(
            \begin{array}{cc}
              \phi^{-1}F/2 & -\phi^{-1} \\
              \phi+1/4F\phi^{-1}F & -F\phi^{-1}/2 \\
            \end{array}
          \right),\quad J_-\sim \left(
            \begin{array}{cc}
              -(\phi^T)^{-1}F/2& -(\phi^T)^{-1} \\
              \phi^T+1/4F(\phi^T)^{-1}F & F(\phi^T)^{-1}/2 \\
            \end{array}
          \right),\]
\[g\sim\left(
         \begin{array}{cc}
           (\phi^{-1})_s & (\phi^{-1})_aF/2 \\
            F(\phi^{-1})_a/2& \phi_s+\frac{1}{4}F(\phi^{-1})_sF \\
         \end{array}
       \right),\quad b\sim \left(
                             \begin{array}{cc}
                               (\phi^{-1})_a & (\phi^{-1})_sF/2 \\
                               F(\phi^{-1})_s/2 & \phi_a+\frac{1}{4}F(\phi^{-1})_aF \\
                             \end{array}
                           \right).
\]

\section{An explicit example on $\mathbb{C}P^1\times\mathbb{C}P^1$}\label{Ex}
In this section, to demonstrate the general theory we have developed we shall construct toric GK structures of symplectic type on the ruled surface $M=\mathbb{C}P^1\times\mathbb{C}P^1$.

Let $M$ be equipped with the symplectic structure
\[\Omega=\frac{\sqrt{-1}}{2}\frac{dz_1\wedge d\bar{z}_1}{(1+|z_1|^2)^2}+\frac{\sqrt{-1}}{2}\frac{dz_2\wedge d\bar{z}_2}{(1+|z_2|^2)^2}.\]
The standard $\mathbb{T}^2$-action
\[(e^{\sqrt{-1}\theta_1},e^{\sqrt{-1}\theta_2})\cdot ([1:z_1], [1:z_2])=([1: e^{\sqrt{-1}\theta_1}z_1], [1:e^{\sqrt{-1}\theta_2}z_2])\]
on $M$ is Hamiltonian relative to $\Omega$. The infinitesimal action is then given by
\[\partial_{\theta_j}=\sqrt{-1}(z_j\partial_{z_j}-\bar{z}_j\partial_{\bar{z}_j}),\quad j=1,2,\]
and the moment map for this action is chosen to be
\[\mu_j=\frac{|z_j|^2}{2(1+|z_j|^2)},\quad j=1,2.\]
The moment polytope $\Delta$ is therefore $[0,1/2]\times [0,1/2]$. Due to Guillemin's formula, the symplectic potential of the standard toric K$\ddot{a}$hler structure in this case is
\[\tau=\frac{1}{2}\sum_{j=1}^2[\mu_j\ln \mu_j+(\frac{1}{2}-\mu_j)\ln(\frac{1}{2}-\mu_j)],\]
whose Hessian $\phi_s$ is
\[\left(
    \begin{array}{cc}
      \frac{1}{4\mu_1(1/2-\mu_1)} & 0 \\
      0 & \frac{1}{4\mu_2(1/2-\mu_2)} \\
    \end{array}
  \right).
\]
Let
\[C=\left(
      \begin{array}{cc}
        0 & c \\
        -c & 0 \\
      \end{array}
    \right),\quad F=\left(
                           \begin{array}{cc}
                             0 & f \\
                             -f & 0 \\
                           \end{array}
                         \right),
\]
where $f\neq 0$ (the case $f=0$ was investigated in \cite{Wang2}). By Example~\ref{ex1}, for the triple $(\tau, C, F)$ to determine a toric GK structure of symplectic type, $f$ must satisfy
\[1-4f^2\mu_1\mu_2(1/2-\mu_1)(1/2-\mu_2)>0,\quad (\mu_1, \mu_2)\in \Delta.\]
The function $1/\det \phi_s=16\mu_1\mu_2(1/2-\mu_1)(1/2-\mu_2)$ takes its maximum $1/16$ when $\mu_1=\mu_2=1/4$. We thus find that $f$ must lie in the open interval $(-8, 8)$.

Now let \[\phi=\phi_s+C=\left(
    \begin{array}{cc}
      \frac{1}{4\mu_1(1/2-\mu_1)} & c \\
      -c & \frac{1}{4\mu_2(1/2-\mu_2)} \\
    \end{array}
  \right),\]
 and consequently
  \[\phi^{-1}=\frac{1}{\det \phi}\left(
                                   \begin{array}{cc}
                                     \frac{1}{4\mu_2(1/2-\mu_2)} & -c \\
                                     c &  \frac{1}{4\mu_1(1/2-\mu_1)} \\
                                   \end{array}
                                 \right),
  \]
  where $\det \phi=\frac{1}{16\mu_1(1/2-\mu_1)\mu_2(1/2-\mu_2)}+c^2$. For later convenience, we introduce some notation:
  \[p:=\frac{1}{16\mu_1(1/2-\mu_1)\mu_2(1/2-\mu_2)},\quad \varrho_j=dz_j/z_j,\quad j=1,2.\]
Note that in the present setting, in the admissible coordinates $\theta, \mu$, the matrix form of $g$ is
\[\frac{1}{\det \phi}\left(
         \begin{array}{cccc}
           \frac{1}{4 \mu_2(1/2-\mu_2)} & 0 & \frac{cf}{2} & 0 \\
           0 &  \frac{1}{4 \mu_1(1/2-\mu_1)} & 0 & \frac{cf}{2} \\
           \frac{cf}{2} & 0 & \frac{\det \phi-\frac{f^2}{4}}{4\mu_1(1/2-\mu_1)} & 0 \\
           0 & \frac{cf}{2} & 0 & \frac{\det \phi-\frac{f^2}{4}}{4\mu_2(1/2-\mu_2)} \\
         \end{array}
       \right).
\]
Similarly, the matrix form of $b$ is
\[\frac{1}{\det \phi}\left(
    \begin{array}{cccc}
      0 & -c & 0 & \frac{f}{8\mu_2(1/2-\mu_2)} \\
      c & 0 & -\frac{f}{8\mu_1(1/2-\mu_1)} & 0 \\
      0 & \frac{f}{8\mu_1(1/2-\mu_1)} & 0 & c(\det\phi+\frac{f^2}{4}) \\
      -\frac{f}{8\mu_2(1/2-\mu_2)} & 0 & -c(\det\phi+\frac{f^2}{4}) & 0 \\
    \end{array}
  \right),
\]
or
\begin{eqnarray*}b&=&\frac{1}{\det \phi}\times[-cd\theta_1d\theta_2+\frac{fd\theta_1d\mu_2}{8\mu_2(1/2-\mu_2)}-\frac{fd\theta_2d\mu_1}{8\mu_1(1/2-\mu_1)}\\&+&c(\det\phi
+\frac{f^2}{4})d\mu_1d\mu_2].\end{eqnarray*}

As was noted in Example~\ref{ex1}, on $\mathring{M}$ $\mathbb{J}_1$ is as well of symplectic type and in particular its pure spinor\footnote{We won't review the spinor description of GC structures here. For this see \cite{Gu0}} is $e^{b'-\sqrt{-1}Q}$, where $b'$ is a real 2-form and $Q$ is a symplectic form (the inverse of $\beta_1$). It can be found that
\[Q=-g(\frac{J_+-J_-}{2})^{-1},\quad b'=-\frac{1}{2}Q(J_++J_-).\]
The matrix of $[(J_+-J_-)/2]^{-1}$ is
\[-\frac{1}{c^2+f^2/4}\times\left(
    \begin{array}{cccc}
      0 & \frac{f}{8\mu_2(1/2-\mu_2)} & 0 & c \\
      -\frac{f}{8\mu_1(1/2-\mu_1)} & 0 & -c & 0 \\
      0 & c(\det\phi+\frac{f^2}{4}) & 0 & -\frac{f}{8\mu_1(1/2-\mu_1)} \\
      -c(\det\phi+\frac{f^2}{4}) & 0 & \frac{f}{8\mu_2(1/2-\mu_2)} & 0 \\
    \end{array}
  \right).
\]
We can obtain the two 2-forms $Q$ and $b'$:
\begin{eqnarray*}Q&=&\frac{1}{c^2+f^2/4}[\frac{f}{2}d\theta_1d\theta_2+\frac{cd\theta_1d\mu_2}{4\mu_2(1/2-\mu_2)}-\frac{cd\theta_2d\mu_1}{4\mu_1(1/2-\mu_1)}\\
&+&\frac{fc^2+f^3/4-pf}{2}d\mu_1d\mu_2],\end{eqnarray*}
\begin{eqnarray*}
b'&=&-(\frac{f}{c^2+f^2/4}-\frac{f}{\det \phi})[\frac{d\theta_1d\mu_2}{8\mu_2(1/2-\mu_2)}-\frac{d\theta_2d\mu_1}{8\mu_1(1/2-\mu_1)}]\\
&+&(\frac{c}{c^2+f^2/4}-\frac{c}{\det \phi})d\theta_1d\theta_2+(-\frac{cp}{c^2+f^2/4}+\frac{cf^2}{4\det \phi})d\mu_1d\mu_2.
\end{eqnarray*}

Finally, let us have a look at the symmetric case, i.e. $c=0$. Note that
\[d\theta_j=-\frac{\sqrt{-1}}{2}(\varrho_j-\bar{\varrho}_j),\quad d\mu_j=\frac{|z_j|^2}{2(1+|z_j|^2)^2}(\varrho_j+\bar{\varrho}_j).\]
By using these formulae, we can find that in terms of the Euclidean coordinates
\[b'-b-\sqrt{-1}Q=\frac{2\sqrt{-1}dz_1\wedge dz_2}{fz_1z_2}-\frac{\sqrt{-1}fd|z_1|^2\wedge d|z_2|^2}{8[(1+|z_1|^2)(1+|z_2|^2)]^2}.\]
Note that the first term of the right hand side corresponds to the $I_0$-holomorphic Poisson structure $\frac{f\sqrt{-1}}{2}z_1z_2\partial_{z_1}\wedge \partial_{z_2}$ while the second term seems to represent the effect of the purely tangent transform.

\section{Generalized Delzant construction}
In \S~\ref{CON}, to construct a toric GK structure, we shall start with a compact toric K$\ddot{a}$hler manifold. For each Delzant polytope $\Delta$, a canonical choice is Delzant's toric K$\ddot{a}$hler manifold $M_\Delta$, which is the K$\ddot{a}$hler reduction of $\mathbb{C}^d$ by a torus action, where $d$ is the number of faces of $\Delta$ of codimension 1 \cite{Del}. If we start with this $M_\Delta$, a nontrivial anti-symmetric matrix $C$ and $F=0$, the toric GK structure thus constructed was similarly interpreted in \cite{Wang2} as obtained from GK reduction of a toric GK structure on $\mathbb{C}^d$ by a torus action. The basic goal of this section is, to some extent, to generalize this result to the general case.

As a first step, we shall apply the construction of Thm.~\ref{cons} to the non-compact manifold $\mathbb{C}^d$ with its standard toric K$\ddot{a}$hle structure. If only $C$ is turned on, this works fairly well and was realized in \cite{Wang2}. If $F$ is also turned on, the situation becomes very subtle. Actually we cannot expect to construct a toric GK structure on the whole of $\mathbb{C}^d$ in the manner of Thm.~\ref{cons}.

Let us describe the standard toric K$\ddot{a}$hler structure on $\mathbb{C}^d$ in the spirit of Abreu-Guillemin theory. $\mathbb{C}^d$ is equipped with the standard symplectic form
\[\Omega'=\frac{\sqrt{-1}}{2}\sum_{j=1}^ddz_j\wedge d\bar{z}_j\]
and the as well standard action of a $d$-dimensional torus $\mathbb{T}^d$:
\[(e^{\sqrt{-1}\theta_1},\cdots,e^{\sqrt{-1}\theta_d})\cdot (z_1,\cdots, z_d)=(e^{\sqrt{-1}\theta_1}z_1,\cdots,e^{\sqrt{-1}\theta_d}z_d).\]
The infinitesimal action is generated by
\[\partial_{\theta_j}=\sqrt{-1}(z_j\partial_{z_j}-\bar{z}_j\partial_{\bar{z}_j}),\quad j=1,2,\cdots, d.\]
This action is Hamiltonian with a moment map $\nu: \mathbb{C}^d\rightarrow (\mathbb{R}^d)^*$, i.e.,
\[\nu(z_1,\cdots, z_d)=\frac{1}{2}(|z_1|^2+2\lambda_1, |z_2|^2+2\lambda_2,\cdots, |z_d|^2+2\lambda_d),\]
or
\[\nu_j=\frac{1}{2}|z_j|^2+\lambda_j,\quad j=1,2,\cdots, d,\]
where $\lambda_j$ are real numbers to be determined by a Delzant polytope (see below).
In terms of admissible coordinates $\theta, \nu$, the metric on $\mathbb{C}^d$ is of the following form:
\[g_0=\sum_{j=1}^d(|z_j|^2(d\theta_j)^2+\frac{(d\nu_j)^2}{|z_j|^2}).\]
Thus the canonical K$\ddot{a}$hler structure is described by the diagonal matrix
\[\phi'_s=\textup{Diag}\{1/|z_1|^2, 1/|z_2|^2,\cdots, 1/|z_d|^2\}.\]
and the corresponding symplectic potential is
\[\tau'=\frac{1}{2}\sum_{j=1}^d(\nu_j-\lambda_j)\ln(\nu_j-\lambda_j).\]

Now to understand the situation we are facing properly, let us see what happens to the case of $d=2$.
\begin{example}\label{ex2}Let us consider $d=2$ and fix \[
C=\left(
    \begin{array}{cc}
      0 & c \\
      -c & 0 \\
    \end{array}
  \right),\quad
F=\left(
                                                 \begin{array}{cc}
                                                   0 & f \\
                                                   -f & 0 \\
                                                 \end{array}
                                               \right),\quad f\neq 0\] In this case,
\[\phi'_s=\left(
          \begin{array}{cc}
            1/|z_1|^2 & 0 \\
            0 & 1/|z_2|^2 \\
          \end{array}
        \right),\]
and
\[\Xi=\phi'_s+\frac{1}{4}F(\phi_s')^{-1}F=\left(
                                      \begin{array}{cc}
                                        \frac{1}{|z_1|^2}-\frac{f^2|z_2|^2}{4} & 0 \\
                                        0 & \frac{1}{|z_2|^2}-\frac{f^2|z_1|^2}{4} \\
                                      \end{array}
                                    \right).
\]
Obviously, $\Xi$ can not be positive-definite on the whole of $\mathbb{C}^2$. Instead, the condition (\ref{posi1}) implies what really matters is the open subset
\[\mathcal{K}_F:=\{(z_1, z_2)\in \mathbb{C}^2||z_1z_2|<\frac{2}{|f|}\}.\]
$\mathcal{K}_F$ is certainly $\mathbb{T}^2$-invariant and we can apply the construction of Thm.~\ref{cons} to $\mathcal{K}_F$.

Now on $\mathcal{K}_F\cap (\mathbb{C}^*)^2$ we have
\[\left(
    \begin{array}{c}
      \zeta^+_1 \\
      \zeta^+_2 \\
    \end{array}
  \right)=\left(
            \begin{array}{c}
              d\theta_1 \\
              d\theta_2 \\
            \end{array}
          \right)-\frac{1}{2}\left(
                               \begin{array}{cc}
                                 0 & f \\
                                 -f & 0 \\
                               \end{array}
                             \right)\left(
                                      \begin{array}{c}
                                        d\nu_1 \\
                                        d\nu_2 \\
                                      \end{array}
                                    \right)=\left(
                                              \begin{array}{c}
                                                d\theta_1-\frac{1}{2}fd\nu_2 \\
                                                d\theta_2+\frac{1}{2}fd\nu_1 \\
                                              \end{array}
                                            \right),
\]
and similarly,
\[\left(
    \begin{array}{c}
      \zeta^-_1 \\
      \zeta^-_2 \\
    \end{array}
  \right)=\left(
                                              \begin{array}{c}
                                                d\theta_1+\frac{1}{2}fd\nu_2 \\
                                                d\theta_2-\frac{1}{2}fd\nu_1 \\
                                              \end{array}
                                            \right).\]
It's straightforward to find that
\[dz_1-\frac{1}{2}(c+\frac{\sqrt{-1}f}{2})z_1d|z_2|^2,\quad dz_2+\frac{1}{2}(c+\frac{\sqrt{-1}f}{2})z_2d|z_1|^2\]
form a frame of the $J_+$-holomorphic cotangent bundle of $\mathcal{K}_F\cap (\mathbb{C}^*)^2$. Similarly,
 \[dz_1+\frac{1}{2}(c+\frac{\sqrt{-1}f}{2})z_1d|z_2|^2,\quad dz_2-\frac{1}{2}(c+\frac{\sqrt{-1}f}{2})z_2d|z_1|^2\]
 form a frame of the $J_-$-holomorphic cotangent bundle of $\mathcal{K}_F\cap (\mathbb{C}^*)^2$. The biHermitian metric $g$, as a linear map, is of the following form relative to the admissible coordinates $\theta_j, \nu_j$:
 \[\frac{1}{1+c^2|z_1z_2|^2}\times \left(
     \begin{array}{cccc}
       |z_1|^2 & 0 & cf|z_1z_2|^2 & 0 \\
       0 & |z_2|^2 & 0 & cf|z_1z_2|^2 \\
       cf|z_1z_2|^2 & 0 & \frac{1+c^2|z_1z_2|^2-\frac{1}{4}f^2|z_1z_2|^2}{|z_1|^2} & 0 \\
       0 & cf|z_1z_2|^2 & 0 & \frac{1+c^2|z_1z_2|^2-\frac{1}{4}f^2|z_1z_2|^2}{|z_2|^2} \\
     \end{array}
   \right).\]
 It's easy to find that $J_\pm$ and $g$ extend smoothly to $\mathcal{K}_F$ and on $\mathcal{K}_F\backslash(\mathcal{K}_F\cap (\mathbb{C}^*)^2)$ the complex structures $J_\pm$ degenerate into the canonical one on $\mathbb{C}^2$ and $g$ degenerates into the canonical Euclidean metric on $\mathbb{C}^2$. This shows that the GK structure we have constructed is well-defined on $\mathcal{K}_F$.
\end{example}

Let $\mathbb{C}^d$ be equipped with its standard toric K$\ddot{a}$hler structure. In general, we have
\begin{theorem}\label{del}
 Given two $d\times d$ anti-symmetric real matrices $C$ and $F$, let $\mathcal{K}_F$ be the $\mathbb{T}^d$-invariant nonempty open subset of $\mathbb{C}^d$ where the condition (\ref{posi1}) is satisfied. Then the pair $(C, F)$ gives rise to a toric GK structure of symplectic type on $\mathcal{K}_F$ in the manner of Thm.~\ref{cons}.
\end{theorem}
\begin{proof}
$\mathcal{K}_F$ is nonempty because it obviously contains a neighbourhood of $0\in \mathbb{C}^d$. The construction certainly works for $\mathcal{K}_F\cap (\mathbb{C}^*)^d$. We only need to show that this construction extends smoothly to $\mathcal{K}_F$. The strategy of the proof of Thm.~\ref{cons} can be adapted without any essential modification to complete the proof. Thus we omit the details.
\end{proof}

Now let us turn to a brief review of Delzant's construction of $M_\Delta$. If the Delzant polytope $\Delta$ (of dimension $n$) is defined by
\[l_j(x):=(u_j, x)\geq \lambda_j,\quad j=1,2,\cdots, d,\]
then there is the linear map $\varsigma: \mathbb{R}^d\rightarrow \mathbb{R}^n$, $e_j\mapsto u_j$, where $\{e_j\}$ is the standard basis of $\mathbb{R}^d$. Let $\mathfrak{n}$ be the kernel of $\varsigma$. Then we have the short exact sequence
\begin{equation}0\longrightarrow \mathfrak{n}\stackrel{\iota}\longrightarrow \mathbb{R}^d \stackrel{\varsigma}\longrightarrow \mathbb{R}^n \longrightarrow0,\label{seq}\end{equation}
where each middle term should be understood as the Lie algebra of the corresponding torus and $\iota$ is the natural inclusion map. This sequence then lifts to the level of Lie groups:
\begin{equation}0\longrightarrow N \longrightarrow \mathbb{T}^d \longrightarrow \mathbb{T}^n \longrightarrow0.\label{seq2}\end{equation}
One applies K$\ddot{a}$hler reduction to the $N$-action on $\mathbb{C}^d$ and the K$\ddot{a}$hler quotient is precisely $M_\Delta$ equipped with the residual Hamiltonian $\mathbb{T}^n$-action. The symplectic potential of $M_\Delta$ is given by Guillemin's formula (\ref{formu}). 

Now let $C, F$ be two $n\times n$ constant anti-symmetric matrices such that Thm.~\ref{cons} applies. We can as well lift them to the level of $\mathbb{C}^d$ and $\mathbb{T}^d$ as follows. From (\ref{seq}) we have an induced map $\varsigma_\wedge: \wedge^2 \mathbb{R}^d\rightarrow \wedge^2 \mathbb{R}^n$. Intrinsically understood, $C$ and $F$ are skew-symmetric bilinear functions on $(\mathbb{R}^n)^*$ in which the moment map $\mu$ on $M_\Delta$ takes values, or in other words, $C$ and $F$ live in $\wedge^2 [(\mathbb{R}^n)^*]^*\cong \wedge^2 \mathfrak{t}^n$. Let $C', F'\in \wedge^2 \mathbb{R}^d$ such that $\varsigma_\wedge(C')=C$ and $\varsigma_\wedge(F')=F$. Then $C'$ and $F'$ are skew-symmetric bilinear function on $(\mathbb{R}^d)^*$ in which the moment map $\nu$ on $\mathbb{C}^d$ takes values or equivalently $C', F'\in \wedge^2 \mathfrak{t}^d$. Let $\mathcal{K}_{F'}$ be the toric GK manifold determined by the canonical toric K$\ddot{a}$hler structure on $\mathbb{C}^d$ and the pair $(C', F')$ in the manner of Thm.~\ref{cons}.
\begin{theorem} The GK quotient (in the sense of \cite{LT}) of $\mathcal{K}_{F'}$ under the Hamiltonian action of $N$ is a $\mathbb{T}^n$-invariant GK open submanifold of the toric GK manifold $M_\Delta$ whose toric GK structure is determined in Thm.~\ref{cons} by the canonical toric K$\ddot{a}$hler structure and the pair $(C, F)$. 
\end{theorem}

\begin{example}Before we can prove the theorem, let us demonstrate that generally we cannot expect that $\mathcal{Z}/N$ is the whole of $M_\Delta$. As in Example~\ref{ex2}, take $d=2$. Let the Delzant polytope be simply $[0, 1/2]$, which is described by the inequalities $\mu_1\geq 0$ and $-\mu_2\geq -1/2$. Then $\lambda_1=0$, $\lambda_2=-1/2$ and $N=S^1$ is the diagonal of $\mathbb{T}^2$. The moment map $\nu_N$ is then
\[\nu_N=|z_1|^2/2+|z_2|^2/2-1/2,\]
and in this case $M_\Delta$ is the quotient of $S^3$ by the diagonal $S^1$-action. However, if $F'=\left(
                                                                                                    \begin{array}{cc}
                                                                                                      0 & f' \\
                                                                                                      -f' & 0 \\
                                                                                                    \end{array}
                                                                                                  \right)
\neq 0$, we know from Example~\ref{ex2} that points in the zero-level set of $\nu_N$ in $\mathcal{K}_{F'}$ should satisfy
\[|z_1|^2+|z_2|^2=1,\quad |z_1z_2|<\frac{ 2}{|f'|}.\]
Then $\nu_N^{-1}(0)$ is not the whole of the 3-sphere unless $f'$ is sufficiently small. Consequently, the quotient is generally only an open subset of $S^2$.
\end{example}
\begin{proof}One should notice first that the present situation does fit in well with the general formalism developed in \cite{LT}. So we do have a GK quotient by the (extended) action of $N$. Additionally, the story is rather classical on the symplectic side--it is in essence the symplectic reduction and the quotient GK structure is consequently of symplectic type. There is of course a residual $\mathbb{T}^n$-action on the quotient, which preserves the quotient GK structure. The point here is to see the quotient toric GK structure is really characterized by the canonical toric K$\ddot{a}$hler structure on $M_\Delta$ and the pair $(C, F)$.

The details of the proof is an application of metric reduction developed in \cite{Ca} \cite{wang0}, which is mainly an account of GK reduction in terms of more traditional notions like Riemmanian metrics and complex structures. Let us review this briefly in our present setting.

Note that here the moment map $\nu_N$ of the $N$-action is the restriction of $\nu$ on $\mathfrak{n}$. It is well-known that $0$ is a regular value of $\nu_N$ (for example see \cite{Gul}). Let $\mathcal{Z}:=\nu^{-1}_N(0)$. Then $N$ acts freely on $\mathcal{Z}$ and $\mathcal{Z}/N$ is an open subset of $M_\Delta$. We can choose another basis $\{f_i\}$ of $\mathfrak{t}^d$ such that $\{f_1,\cdots, f_{d-n}\}$ is a basis of $\mathfrak{n}$\footnote{The choice of course has also changed the admissible coordinates on $\mathcal{K}_{F'}$ and many other things depending on them, but by abuse of notation we won't bother to introduce new ones.}. Denote the fundamental vector filed associated to $f_i$ by $X_i$. Then correspondingly the components $\nu_i$, $i=1,\cdots, d-n$ are actually the components of $\nu_N$. The extended action of $\mathfrak{n}$ on $\mathcal{K}_{F'}$ is generated by
\[X_i-b'(X_i),\quad i=1,2,\cdots, d-n,\]
where $b'=-1/2\Omega'(J_+'-J_-')$ and $J_\pm'$ are the underlying complex structures of the GK structure on $\mathcal{K}_{F'}$. Note that $q: \mathcal{Z}\rightarrow \mathcal{Z}/N$ is a principal $\mathbb{T}^{d-n}$-bundle. Then there are \emph{three} horizontal distributions determined by
\[\mathcal{D}_\pm=\{Y\in T\mathcal{Z}|g'(Y, X_i)\pm b'(Y,X_i)=0,i=1,2,\cdots, d-n\},\]
and by
\[\mathcal{D}_0=\{Y\in T\mathcal{Z}|\Omega'(Y, I_0'X_i)=0,i=1,2,\cdots, d-n\}\]
where $g'=-1/2\Omega'(J_+'+J_-')$ is the biHermitian metric on $\mathcal{K}_{F'}$ and $I_0'$ is the canonical complex structure on $\mathcal{K}_{F'}$. Then $\mathcal{D}_\pm$ are $J_\pm'$-invariant respectively and $\mathcal{D}_0$ is $I_0'$-invariant. Identifying $\mathcal{D}_\pm/N$ and $\mathcal{D}_0/N$ with $T(\mathcal{Z}/N)$ then produces the reduced complex structures $J_\pm$ and $I_0$ on $\mathcal{Z}/N$.

Now in our present setting, it's evident that we only need to prove the theorem for the open dense subset $\mathcal{K}_{F'}\cap (\mathbb{C}^*)^d$. In the following computation we mainly use the admissible coordinates $\theta_i, \nu_i$ associated to $I_0'$.

Let us find what $J_+$ is. $Y\in \mathcal{D}_+$ should satisfy
\[d\nu_i(Y)=0,\quad i=1,\cdots, d-n\]
and
\[g'(Y, X_i)+b'(Y,X_i)=-\Omega'(J_+'Y, X_i)=-d\nu_i(J_+'Y)=0, \quad \quad i=1,\cdots, d-n.\]
It is not hard to find that $\mathcal{D}_+$ is generated by
\[\partial_{\nu_i}-\frac{1}{2}\sum_{j=1}^dF_{ji}X_j,\quad J_+'(\partial_{\nu_i}-\frac{1}{2}\sum_{j=1}^dF_{ji}X_j),\quad i=d-n+1,\cdots, d.\]
Actually note that from the matrix forms listed at the end of \S~\ref{CON} we can obtain
\begin{eqnarray*}
J_+'(\partial_{\nu_i}-\frac{1}{2}\sum_{j=1}^dF_{ji}X_j)&=&\sum_{k=1}^d[\phi'+\frac{1}{4}F'(\phi')^{-1}F']_{ki}X_k-\sum_{k=1}^d(F'(\phi')^{-1}/2)_{ki}\partial_{\nu_k}\\
&+&\frac{1}{2}\sum_{j,k=1}^dF_{ji}[(\phi')^{-1})^{kj}\partial_{\nu_k}-((\phi')^{-1}F/2)_{kj}X_k]\\
&=&\sum_{k=1}^d(\phi')_{ki}X_k,
\end{eqnarray*}
where $\phi'=\phi_0+C$. With the projection $q_*$, the above computation means
\begin{equation}J_+(\partial_{\nu_i}-\frac{1}{2}\sum_{j=d-n+1}^dF_{ji}q_*(X_j))=\sum_{k=d-n+1}^d(\phi')_{ki}q_*(X_k),\quad i=d-n+1,\cdots, d.\label{j+}\end{equation}
A similar result holds for $J_-$.

The expression of $I_0$ is even easier: $Y\in \mathcal{D}_0$ should satisfy
\[d\nu_i(Y)=0,\quad d\nu_i(I_0'Y)=0,\quad i=1,\cdots, d-n,\]
and we can find $\mathcal{D}_0$ is generated by $\partial_{\nu_i}, I_0'\partial_{\nu_i}$, $i=d-n+1,\cdots, d$. With the projection $q_*$, it is obtained that
\begin{equation}I_0(\partial_{\nu_i})=\sum_{j=d-n+1}^d(\phi'_s)_{ji}q_*(X_j), \quad i=d-n+1, \cdots, d.\label{j0}\end{equation}

Taking the dual form of Eq.~(\ref{j+}) and Eq.~(\ref{j0}), we finally find that
\[d\theta_i^+=d\theta_i-\frac{1}{2}\sum_{j=d-n+1}^dF_{ji}d\nu_j, \quad i=d-n+1,\cdots, d\] gives rise to the admissible connection associated to $J_+$, and
\[J_+^*d\theta_i^+=\sum_{j=d-n+1}^d(\phi')_{ij}d\nu_j,\quad i=d-n+1,\cdots, d.\]
Obviously, a similar result holds for $J_-$. Compared with the proof of Thm.~\ref{GKG}, these precisely imply that the canonical K$\ddot{a}$hler structure on $M_\Delta$ and the matrices $C=\varsigma_\wedge(C')$, $F=\varsigma_\wedge(F')$ parameterize the toric GK structure on $\mathcal{Z}/N$ as expected. This completes the proof.
\end{proof}
\section{Appendix}
We collect some facts concerning matrices here. These are elementary but frequently (maybe implicitly) used in the main text of this article. For a matrix $A$, let $A^T$ be its transpose and $A_s$, $A_a$ its symmetric and anti-symmetric parts respectively. In the following facts except the last one, let $A$ be an $n\times n$ invertible matrix and $B$ its inverse.

\textbf{Fact I}. \[A_sB_s+A_aB_a=B_sA_s+B_aA_a=\textup{I},\]
\[A_sB_a+A_aB_s=B_sA_a+B_aA_s=0.\]

\textbf{Fact II.} $A_s$ is invertible if and only if $B_s$ is invertible; in particular, if $A_s$ is positive-definite, then so is $B_s$.

\textbf{Fact III}. $AB_sA^T=A_s$ and $AB_aA^T=-A_a$. In particular, we have \[B^T(B_s)^{-1}B=(A_s)^{-1}.\]

\textbf{Fact IV}. If $A, B$ are both $n\times n$ positive-definite symmetric matrices, and $A\geq B$, then $B^{-1}\geq A^{-1}$.
\begin{proof}
Let $0\neq v\in \mathbb{R}^n$. Then $v^TAv\geq v^TBv$.
Let $w=\sqrt{A}v$. Then
\[w^Tw\geq w^T A^{-1/2}BA^{-1/2}w.\]
Since $A$ is invertible, $w$ can be an arbitrary vector in $\mathbb{R}^n$. Then the above inequality implies that all eigenvalues of $A^{-1/2}BA^{-1/2}$ are $\leq 1$ and consequently all eigenvalues of $A^{1/2}B^{-1}A^{1/2}$ are $\geq 1$. Therefore,
\[v^T v\leq v^TA^{1/2}B^{-1}A^{1/2}v,\]
or equivalently $w^TA^{-1}w\leq w^TB^{-1}w$. That's to say $B^{-1}\geq A^{-1}$.
\end{proof}
\textbf{Fact V}. The map $A\mapsto \sqrt{A}$ on the space of $n\times n$ nonnegative-definite symmetric matrices is continuous. This is only a simple conclusion of the more general consideration in \cite{Chen}.
\section*{Acknowledgemencts}
This study is supported by the Natural Science Foundation of Jiangsu Province (BK20150797). The manuscript is prepared during the author's stay in the Department of Mathematics at the University of Toronto and this stay is funded by the China Scholarship Council (201806715027). The author also thanks Professor Marco Gualtieri for his invitation and hospitality.

\end{document}